\documentclass{siamltex}
\usepackage{mathrsfs}
\usepackage{color}
\usepackage{amscd}
\usepackage{threeparttable,booktabs}
\usepackage{algorithm,algorithmic}
\usepackage{amsmath,amsfonts,amssymb,graphicx}%

\renewcommand{\algorithmicrequire}{\textbf{Input:}}   

\renewcommand{\algorithmicensure}{\textbf{Output:}}  

\def \aa{\mathfrak a}
\def \rea{\mathfrak a}
\def \bb{\mathfrak b}
\def \bone{\mathbf 1}

\def\RR{\mathbb{R}}
\def\CC{\mathbb{C}}

\def\ba{{\pmb \alpha}}
\def\Range{{\rm Range}}
\def\range{{\rm Range}}
\def\dist{{\rm dist}}

\newcommand{\Ham}{\mathcal{H}}
\newcommand{\Sym}{\mathcal{S}}

\def\name{{\rm ILRSI}}

\newtheorem{remark}[theorem]{Remark}

\newtheorem{example}[theorem]{Example}

\title{A new subspace iteration method for the algebraic Riccati equation %
 \thanks{Version of July 10, 2013.}}

\author{Yiding Lin\thanks{School of Mathematical Sciences, Xiamen University, China and Dipartimento di Matematica, Universit{\`a} di Bologna, Bologna, Italy({\tt Yiding.Lin@gmail.com}).}
        \and Valeria Simoncini\thanks{Dipartimento di Matematica, Universit{\`a} di Bologna,
Piazza di Porta S. Donato, 5, 40127 Bologna, Italy   ({\tt valeria.simoncini@unibo.it}).}}

\begin{document}
\bibliographystyle{plain}
\maketitle

\begin{abstract}
We consider the numerical solution of the continuous algebraic
Riccati equation 
$A^*X+XA-XFX+G=0$, with $F=F^*, G=G^*$ of low rank and $A$ large and sparse.
We develop an algorithm for the low rank approximation of $X$
by means of an invariant subspace iteration on a function of the associated
Hamiltonian matrix. We show that the sought after approximation 
can be obtained by a low rank update, in the style of the well known ADI iteration
for the linear equation, from which the new method inherits many algebraic properties.
Moreover, we establish new insightful matrix relations with emerging projection-type
methods, which will help increase our understanding of this latter class of solution
strategies.
\end{abstract}



\pagestyle{myheadings}
\thispagestyle{plain}

\markboth{Yiding Lin and Valeria Simoncini}{Subspace iteration for the algebraic Riccati equation} 


\section{Introduction}

We are interested in the numerical solution of the continuous algebraic
Riccati equation 
\begin{equation}\label{eqn:main}
A^*X+XA-XFX+G=0, \quad F=F^*, G=G^*,
\end{equation}
where $A\in\RR^{n\times n}$ has large dimensions, $F, G \in \RR^{n\times n}$ have
 low rank, and $X$ is the unknown matrix to be approximated\footnote{%
We consider real matrices
because the typical applications we address all have real data. Nonetheless, the method we are going to introduce is well defined also in the complex field.}. 
Here and in the following, $M^*$ denotes the conjugate transpose of the matrix $M$.
 We shall assume
that $A$ is stable, that is its eigenvalues all have strictly negative real part.
The quadratic matrix equation in (\ref{eqn:main}) has a dominant role in the solution and
analysis of
optimal control problems associated with dynamical systems, and it has attracted the
interest of many researchers both for its elegance and its timeliness in applied
field; we refer the reader to, e.g., \cite{Lancaster.Rodman.95},\cite{Antoulas.05},%
\cite{KGMP.03},\cite{Benner2005a},\cite{Schilders2008},\cite{Bittanti.Laub.Willems.91}.

A matrix $X$ solution to (\ref{eqn:main}) is such that the columns of the matrix 
\begin{eqnarray}\label{eqn:IX}
\begin{bmatrix} I_n \\ X\end{bmatrix} ,
\end{eqnarray}
where $I_n\in\RR^{n\times n}$ is the identity matrix,
generate an invariant subspace of the Hamiltonian matrix (\cite{Lancaster.Rodman.95})
$$
\Ham=\begin{bmatrix}A &-F \\ -G &-A^*\end{bmatrix} .
$$
In particular,
we assume that the eigenvalues of $\Ham$ satisfy
\begin{equation}\label{eqn:csplitting}
\begin{aligned}
&\Re(\lambda_1)\leq \Re(\lambda_2)\leq \ldots\leq \Re(\lambda_n)< 0 <\Re(\lambda_{n+1})\leq \Re(\lambda_{n+2})\leq \ldots\leq \Re(\lambda_{2n})\\
\end{aligned} ,
\end{equation}
so that, in particular, no purely imaginary eigenvalues arise.
We look for an approximation to the extremal solution $X_{+}$ of (\ref{eqn:main}),
associated with all eigenvalues of $\Ham$ with negative real part \cite{Lancaster.Rodman.95}. 
Such solution is called a {\it stabilizing} solution, being such that
the matrix $A-FX_+$ is stable.

Many numerical procedures have been explored for solving the quadratic
matrix equation (\ref{eqn:main}), see, e.g., \cite{Binietal.book.12} for a thorough
survey, however few can address the case when $A$ has large 
dimensions \cite{Binietal.book.12},\cite{Benner.Li.Penzl.08},\cite{Benner.Saak.10}. 
In this case, usually a symmetric low rank approximation matrix
is sought, in the form of the product of two matrices, such as $\widehat X = U U^*$,
with $U$ having few columns. Such approach avoids storing the full $n\times n$
matrix, which would be prohibitive for large $n$.
Among these strategies, is the class of exact and inexact Newton methods: Newton's
iteration applied to (\ref{eqn:main}) can be conveniently rewritten so as 
to update the low rank approximate solution and its rank at each iteration.
The approach requires the (in)exact solution of a linear matrix equation 
at each iteration \cite{Kleinman_68},\cite{FHS.09}, which is performed
by means of iterative methods, such as ADI or projection methods;
we refer to \cite{Benner.Saak.survey13} for a very recent survey.
For large matrices stemming from sufficiently regular differential control problems, 
Newton strategies based on hierarchical matrices and nonlinear multigrid methods
have also shown to be effective
\cite{Grasedyck.Hackbusch.Khoromskij.03}, \cite{Grasedyck.08}.

Another class of methods has recently emerged as a
competitive alternative to nonlinear (Newton) solvers: the general approach consists in
extending well established projection type methods to the quadratic case,
with no significant modifications 
\cite{Jbilou_03},\cite{Heyouni.Jbilou.09},\cite{BHJ2011},\cite{Simoncini.Szyld.Monsalve.13}.
Although projection methods have gained increasing popularity in the linear
case, with thoroughly analyzed theoretical properties (\cite{Simoncini.survey.13}), 
their exploration
in the quadratic case has only recently started, and much of their properties
remains to be uncovered.

A less exercised class of methods is given by the doubling algorithm,
which was recently explored in the Riccati context in \cite{Lin.Chu.Lin.Weng.13};
however its memory and computational requirements have not been fully analyzed
for large nonsymmetric problems.

All these approaches attack (\ref{eqn:main}) as a quadratic equation.
We take a different viewpoint, which  consists in approximating $X$ in the second block of 
the matrix in (\ref{eqn:IX}), whose columns span an 
invariant subspace of $\Ham$. Such strategy is quite popular in the
small scale case, when an explicit possibly structure-preserving eigendecomposition
may be determined; see, e.g.,
 \cite{Byers.87},\cite{Arnold.Laub.84},\cite{Laub1991},\cite{Benner1997} and the
extensive treatment in \cite{Binietal.book.12}. A possible adaptation
to the large scale setting was recently proposed 
in \cite{Amodei.Buchot.10}, where an approximation of the form
$X_k=Z W Y^*$ was derived, stemming from the approximation of selected
stable eigenpairs of $\Ham$.

The aim of this paper is to develop an algorithm for the approximation of $X$
by means of an invariant subspace iteration on a function of the matrix $\Ham$
\cite{Laub1991}.
Typically, subspace iteration methods are based on $\Ham$. Here we consider
a subspace iteration method with a transformed matrix obtained using a
Cayley transformation. For $\alpha$ so that $\Ham + \alpha I$ is nonsingular,
the Cayley transformation is given by 
\begin{eqnarray}\label{eqn:Cayley}
\Sym(\alpha)=(\Ham + \alpha I)^{-1}(\Ham -\overline{\alpha} I) ,
\end{eqnarray}
and it is usually employed in the Riccati equation context for
accelerating the computation of the Schur form by 
means of a QR iteration \cite[p.133]{Binietal.book.12}, \cite{Byers.86}.
As a consequence of the transformation, 
the property (\ref{eqn:csplitting}) transforms into 
$|\sigma_1|\geq |\sigma_2|\geq \ldots\geq |\sigma_n|> 1 
>|\sigma_{n+1}|\geq |\sigma_{n+2}|\geq \ldots\geq |\sigma_{2n}|$, 
for the eigenvalues $\sigma_j$ of $\Sym(\alpha)$, and 
the columns of $[I_n;X_+]$  span the invariant subspace of $\Sym(\alpha)$
associated with the $n$ eigenvalues largest in modulus. The transformation
thus provides a more natural setting for a subspace iteration.
We will show that whenever $F$ and $G$ are positive semidefinite
and have low rank, such iteration can be written in terms of a fixed point
recurrence
in the {\it low rank} approximation matrix $X_k$, and a low rank update can be
performed.  To the best
of our knowledge, this iteration appears to be new, in particular 
with the simplification obtained in the low rank case. From our derivation
it readily follows  that this
novel approach {coincides with the ADI method  in the linear case, namely
whenever $F=0$}, thus showing that ADI may be bonded to
a subspace iteration method. 
The proposed method depends on parameters that can be deduced from known
properties of the problem, or estimated a-priori. In that respect, the method
inherits the properties of its linear counterpart ADI.

We will also derive relations between
the new subspace iteration and projection methods 
for the Riccati equation
that use the rational Krylov subspace (RKSM). 
These results provide new insights in the understanding of the
convergence properties of  RKSM when directly applied to (\ref{eqn:main}).

{We emphasize that our developments provide a new and insightful 
matrix framework
that, on the one hand,  will allow us to bridge the gap between methods for
two closely related linear and quadratic equations and, on the other hand,
will be a first step ahead in the understanding of projection 
methods for (\ref{eqn:main}), not based on the Newton method for large scale
problems.}

The following notation will be used throughout the manuscript. 
$ F \succ 0$ ($F \succeq 0$) will
denote a Hermitian and positive (semi-)definite matrix $F$.
The Euclidean norm will be used for vectors, and the associated induced 
norm for matrices, denoted by $\|\cdot \|$, together with the
Frobenius norm, denoted by $\|\cdot \|_F$. 
The notation diag($d$) and blkdiag($D_1,D_2$) will be used to
denote a diagonal matrix with the entries of the vector $d$ on the diagonal, and
a block diagonal matrix with block diagonal entries $D_1, D_2$, respectively.
{We will use Matlab (\cite{matlab7}) notation for matrices
and their subblocks whenever possible.}

\section{A subspace iteration with Cayley transformation}
Given $X_0\in\RR^{n\times n}$ and 
the parameters $\alpha_k$, $k=1,2,\ldots$ with $\Re(\alpha_k)>0$, such
that $(\Ham +\alpha_k I)$ is invertible\footnote{This condition will be
relaxed in the sequel.}, we 
consider the following iteration to compute a sequence of approximations $X_1, X_2, \ldots, 
X_k, \ldots$ to $X_+$.

\hskip 0.2in For $k=1,2, \ldots $

\hskip 0.4in  Compute
\begin{eqnarray}\label{eqn:basic_rec}
\begin{bmatrix}M_k\\N_k\end{bmatrix}:=
\Sym(\alpha_k) \begin{bmatrix}I\\X_{k-1}\end{bmatrix} \qquad 
\mbox{(with $\Sym(\alpha_k)$ as in (\ref{eqn:Cayley}))}
\end{eqnarray}

\hskip 0.4in  $X_k:=N_kM_k^{-1}$

\hskip 0.2in end

The iteration breaks down if  $M_k$ is singular at some iteration. In the
following we shall find a sufficient condition that ensures all $M_k$'s are nonsingular; 
we will also show that
this condition can be easily satisfied when, for instance, $X_0$ is chosen to be
the zero matrix and $F, G$ are positive semidefinite.

By eliminating the intermediate matrices $M_k, N_k$, the recursion above
can be rewritten as a fixed point iteration with the recurrence matrix
$X_k$. This will allow us to derive some crucial properties of the
approximate solution. 
To be able to write down the fixed point iteration, we need to 
express the statement  in (\ref{eqn:basic_rec}) in a more explicit way.
For {any} $\alpha\in\CC$ {such that $\Ham+\alpha I$
is nonsingular}, let
\begin{equation*}
\Ham+\alpha I  
=\begin{bmatrix}A+\alpha I &-F \\ -G &-A^*+\alpha I \end{bmatrix} .
\end{equation*}
If $A+\alpha I$ is nonsingular, then the
Schur complement $S_1(\alpha):= (-A^*+\alpha I) -G (A+\alpha I)^{-1} F$ is
also nonsingular. Analogously, $-A^*+\alpha I$ nonsingular implies
$S_2(\alpha):=(A+\alpha I) - F (-A^*+\alpha I)^{-1}G$ nonsingular.
To simplify the notation, we shall often omit the dependence of $S_1, S_2$
on $\alpha$. It can be readily verified that
\begin{eqnarray}\label{eqn:schurequal}
S_2^{-1}F (-A^*+\alpha I)^{-1}&=&(A+\alpha I)^{-1}FS_1^{-1} , \label{eqn:schurequal1}\\
S_1^{-1}G (A+\alpha I)^{-1}&=&(-A^*+\alpha I)^{-1}GS_2^{-1} . \label{eqn:schurequal2}
\end{eqnarray}
For later use, we notice that we can write
\begin{equation*}
\begin{aligned}
(\Ham+\alpha I)^{-1}
&=\begin{bmatrix}
S_2^{-1} &S_2^{-1}F(-A^*+\alpha I)^{-1} \\ S_1^{-1}G(A+\alpha I)^{-1} &S_1^{-1}
\end{bmatrix}\\
&=\begin{bmatrix}
S_2^{-1} &(A+\alpha I)^{-1}F S_1^{-1} \\ (-A^*+\alpha I)^{-1}G S_2^{-1} &S_1^{-1}
\end{bmatrix} .\\
\end{aligned}
\end{equation*}


At the $k$th iteration, 
let $\alpha_k = \aa_k + \imath \bb_k$, with $\aa_k, \bb_k\in\RR$; this
notation will be used throughout the paper. In particular, from now on we
shall assume that $\aa_k >0$ for all $k$.
Since $\Sym(\alpha_k) = I- 2 \aa_k (\Ham+ \alpha_kI)^{-1}$, we can write
the product in (\ref{eqn:basic_rec}) as follows
\begin{equation*}
\begin{aligned}
\Sym(\alpha_k) \begin{bmatrix}I\\X_{k-1}\end{bmatrix}
&=(I-2\aa_k(\Ham+ \alpha_k I)^{-1})\begin{bmatrix}I\\X_{k-1}\end{bmatrix}\\
&=\begin{bmatrix}I-2\aa_k S_2^{-1} -2\aa_k S_2^{-1}F (-A^*+\alpha_k I)^{-1}X_{k-1} \\ 
-2\aa_k S_1^{-1}G(A+\alpha_k I)^{-1} +(I-2\aa_k S_1^{-1})X_{k-1}\end{bmatrix},\\
\end{aligned}
\end{equation*}
so that the next iterate can be written by means of a fixed point iteration as follows,
\begin{eqnarray}\label{eqn:Xnew}
X_{k}&=&[-2\aa_k S_1^{-1}G(A+\alpha_k I)^{-1} +(I-2\aa_k S_1^{-1})X_{k-1}] \cdot \\
&&\hskip 0.6in [I-2\aa_k S_2^{-1} -2\aa_k S_2^{-1}F (-A^*+\alpha_k I)^{-1}X_{k-1}]^{-1}.
\nonumber
\end{eqnarray}

Notice that because of (\ref{eqn:schurequal1}) and (\ref{eqn:schurequal2}), 
it would be possible
to write the iteration in four possible different but mathematically
equivalent ways.

\begin{remark}
{\rm
If the non-linear term vanishes, that is if $F=0$, 
then the Riccati equation becomes the (linear)
Lyapunov equation $G+A^*X+XA=0$. In this case, it can be
readily noticed that the fixed point iteration in (\ref{eqn:Xnew}) coincides
with the ADI recursion for solving the Lyapunov equation; see, e.g., \cite[formula (4.1)]{Li2002}.
We will return to this correspondence in later sections.\endproof
}
\end{remark}

\section{Properties of the approximate solution}

In this section we analyze the existence of the approximate solution at each step $k$
of the subspace iteration, with $X_k$ obtained as in 
(\ref{eqn:Xnew}).
\begin{theorem}\label{th:existence}
Assume that $F, G \succeq 0$, and that $A^*-\alpha_k I$ is nonsingular. 
In (\ref{eqn:basic_rec}), assume that for some $k>0$ it holds that $X_{k-1}\succeq 0$. Then 
\begin{enumerate}
\item[$i)$] The matrix $M_k$ is nonsingular. 
\item[$ii)$] The matrix $X_k$ is well defined and satisfies $X_k=X_k^*$.
\end{enumerate}
\end{theorem}
\begin{proof}
From the definition of $M_k$ we have
 \begin{eqnarray}
M_k&=&S_2^{-1} (S_2-2\aa_k I -2\aa_k F (-A^*+\alpha_kI)^{-1}X_{k-1})\nonumber\\
&=&S_2^{-1} (A-\bar \alpha_k I-F(-A^*+\alpha_kI)^{-1}G-2\aa_k F (-A^*+\alpha I)^{-1}X_{k-1})\nonumber\\
&=&S_2^{-1} (A-\bar \alpha_k I)[I- (A-\bar \alpha_k I)^{-1}F(-A^*+\alpha_kI)^{-1}(G+2\aa_k X_{k-1})]\nonumber\\
&=&S_2^{-1} (A-\bar \alpha_k I)[I+ (-A+\bar \alpha_kI)^{-1}F(-A^*+\alpha_kI)^{-1}(G+2\aa_k X_{k-1})]  .
\label{eqn:Mk}
\end{eqnarray}
Then, we observe that the nonzero eigenvalues of
$[(-A+\bar\alpha_kI)^{-1}F(-A^*+\alpha_kI)^{-1}](G+2\alpha X_{k-1})$ are all real and
positive, since the matrix is the product of two Hermitian and nonnegative
definite matrices. 
Therefore the quantity in brackets in (\ref{eqn:Mk})
is nonsingular and the first result follows. 

Since $X_k = N_k M_k^{-1}$, the first result ensures that $X_k$ is well defined.
We only need to show that it is Hermitian, namely
$X_k=X_k^*$, which
is equivalent to showing that $M_k^* N_k = N_k^* M_k$.
Let us write  ${\cal S} = [{\cal M}_1, {\cal M}_2;
{\cal N}_1, {\cal N}_2]$ with
%
%
\begin{equation}
\begin{aligned}
\mathcal {N}_1&:=-2\aa_k(-A^*+\alpha_k I)^{-1}GS_2^{-1},   &\mathcal {N}_2&:=I-2\aa_k S_1^{-1}\\
\mathcal {M}_1&:=I-2\aa_k S_2^{-1}, &\mathcal {M}_2&:=-2\aa_k(A+\alpha_k I)^{-1}FS_1^{-1}, \\
\end{aligned}
\end{equation}
We recall that since ${\cal H}$ is Hamiltonian, ${\cal S}$ is symplectic, so 
that from the definition of symplectic matrix
 it follows (\cite[p.24]{Binietal.book.12}) 
\begin{equation} \label{eqn:mathcalMN}
\mathcal {M}_1^*\mathcal {N}_1=\mathcal {N}_1^*\mathcal {M}_1, \quad
\mathcal {M}_2^*\mathcal {N}_2=\mathcal {N}_2^*\mathcal {M}_2, \quad
\mathcal {M}_2^*\mathcal {N}_1-\mathcal {N}_2^*\mathcal {M}_1=-I.
\end{equation}
Moreover,
$X_{k}=N_kM_k^{-1}=(\mathcal {N}_1+\mathcal {N}_2X_{k-1})
(\mathcal {M}_1+\mathcal {M}_2X_{k-1})^{-1}$.
Together with $X_{k-1}=X_{k-1}^*$, relations (\ref{eqn:mathcalMN})
show  that $M_{k}^*N_{k}=N_{k}^*M_{k}$, so that $X_{k}=X_{k}^*$.
\end{proof}

We note that the second result does not explicitly require that $F$ and $G$ be
positive semidefinite. 
{Moreover, the hypothesis that $A-\alpha_k I$ is nonsingular is
always satisfied for $A$ real and stable, and $\Re(\alpha_k)>0$.}

Next proposition derives a more convenient form for the iterate $X_k$, from which
we can deduce that $X_k$ is positive semidefinite for any $k>0$, if $X_0$ is.

\begin{proposition}\label{prop:lowrank}
Assume $F^*=F$, $G=C^*C$ and that for some $k>0$,  $X_{k-1}$ can be written as
$X_{k-1}=U_{k-1}T_{k-1}^{-1}U_{k-1}^*$ with 
$T_{k-1}$ Hermitian and nonsingular. Suppose $X_{k}$ is well defined and let
\begin{eqnarray}\label{eqn:Tk}
T_{k}\!=\!\!\begin{bmatrix}T_{k-1}&0\\0&2\aa_k I\end{bmatrix}
\! +2\aa_k\!\!\begin{bmatrix}U_{k-1}^*\\C\end{bmatrix} 
(-A+\bar \alpha_k I)^{-1}F(-A^*+\alpha_k I)^{-1}\!\begin{bmatrix}U_{k-1}&C^*\end{bmatrix}.
\end{eqnarray}
If $T_{k}$ is nonsingular, then $X_{k}=U_{k}T_{k}^{-1}U_{k}^*$, where
\begin{equation}\label{eqn:lowrankrecursion}
U_{k}=\begin{bmatrix}(-A^*+\alpha_k I)^{-1}(-A^*-\bar\alpha_k I)U_{k-1},&
-2\aa_k(-A^*+\alpha_k I)^{-1}C^*\end{bmatrix}.
\end{equation}
\end{proposition}
\begin{proof}
Using (\ref{eqn:schurequal}), we write
\begin{equation}
\begin{aligned}
X_k&=
[-2\aa_k S_1^{-1}G(A+\alpha I)^{-1} +(I-2\aa_k S_1^{-1})X_{k-1}]
M_k^{-1} \\
&=[-2\aa_k(-A^*+\alpha I)^{-1}GS_2^{-1} +(I-2\aa_k S_1^{-1})X_{k-1}]
M_k^{-1} , \\
\end{aligned}
\end{equation}
and with $G=C^*C$, we can write
\begin{eqnarray*}
&&[-2\aa_k(-A^*+\alpha I)^{-1}GS_2^{-1} +(I-2\aa_k S_1^{-1})X_{k-1}]=\\
&& \qquad
\underbrace{%
\begin{bmatrix}-2\aa_k(-A^*+\alpha I)^{-1}C^*,&(I-2\aa_k S_1^{-1})U_{k-1}
\end{bmatrix}}_{L}
\underbrace{\begin{bmatrix}CS_2^{-1}\\T_{k-1}^{-1}U_{k-1}\end{bmatrix}}_{R} ,
\end{eqnarray*}
so that $X_{k}=L(RM_k^{-1})$. Since $X_k$ is Hermitian (cf. Proposition~\ref{prop:lowrank}), 
$L$ and $M_k^{-*}R^*$ have the same column space,
therefore there exists $\widetilde{T}$ such that 
$M_k^{-*}R^*\widetilde{T}=L$.

Writing $M_k^{*}L=R^*\widetilde{T}$, it is possible to recover $\widetilde{T}$ explicitly (we
omit the tedious algebraic computations), namely
{\footnotesize
\begin{eqnarray}
\widetilde{T}&=& 
\begin{bmatrix}2\aa_k I & \\  & T_{k-1} \end{bmatrix} + \\
&&\!\!\!\!\!\!
\begin{bmatrix}2\aa_k C(-A+\bar \alpha_k I)^{-1}F(-A^*+\alpha_k I)^{-1}C^*&
4\aa_k^2 C(-A+\bar \alpha_k I)^{-1} (A+\alpha_k I)^{-1}FS_1^{-1}U_{k-1}\\
4\aa_k^2U_{k-1}^*(-A+\bar \alpha_k I)^{-1}FS_2^{-*}(-A^*+\alpha_k I)^{-1}C^*&
-2\aa_k U_{k-1}^*(-A+\bar \alpha_k I)^{-1}FS_2^{-*}(I-2\aa_k S_1^{-1})U_{k-1}\nonumber
\end{bmatrix}.
\end{eqnarray}
}
The symmetry of $\widetilde{T}$ can be obtained  after substituting 
(\ref{eqn:schurequal}) into the (2,1) block, {and using
(\ref{eqn:mathcalMN}) for the (2,2) block}. 

Let $\mathcal {P}:=[I-C(A+\alpha_k I)^{-1}F(-A^*+\alpha_k I)^{-1}C^*]^{-1}C
(A+\alpha_k I)^{-1}F(-A^*+\alpha_k I)^{-1}$.
Applying the Sherman-Morrison-Woodbury formula to $S_1$, we obtain
\begin{equation}
\begin{aligned}
&S_1^{-1}=(-A^*+\alpha_k I)^{-1}+(-A^*+\alpha_k I)^{-1}C^*\mathcal {P}\\
&I-2\aa_k S_1^{-1}=(-A^*+\alpha_k I)^{-1}(-A^*-\bar \alpha_k I)
-2\aa_k(-A^*+\alpha_k I)^{-1}C^*\mathcal {P}.
\end{aligned}
\end{equation}
Hence,
\begin{equation}
\begin{aligned}
L&=\begin{bmatrix}-2\aa_k(-A^*+\alpha_k I)^{-1}C^*,&(I-2\aa_k S_1^{-1})U_{k-1}\end{bmatrix}\\
&=\begin{bmatrix}-2\aa_k(-A^*+\alpha_k I)^{-1}C^*,&(-A^*+\alpha_k I)^{-1}(-A^*-\bar\alpha_k I)U_{k-1}\end{bmatrix}
\begin{bmatrix}I&\mathcal {P}U_{k-1}\\0&I\\\end{bmatrix}\\
&=\begin{bmatrix}(-A^*+\alpha_k I)^{-1}(-A^*-\bar\alpha_k I)U_{k-1},&-2\aa_k(-A^*+\alpha_k I)^{-1}C^*\end{bmatrix}
\begin{bmatrix}0&I\\I&\mathcal {P}U_{k-1}\\\end{bmatrix}\\
&=:U_k\begin{bmatrix}0&I\\I&\mathcal {P}U_{k-1}\\\end{bmatrix}.
\end{aligned}
\end{equation}
Explicit computation gives $T_{k}$ in (\ref{eqn:lowrankrecursion}) (explicit details are omitted):
\begin{equation}
\begin{bmatrix}-U_{k-1}^*\mathcal {P}^*&I\\I&0\\\end{bmatrix}\widetilde{T}
\begin{bmatrix}-\mathcal {P}U_{k-1}&I\\I&0\\\end{bmatrix} \equiv T_{k}.
\end{equation}
Note that $T_{k}$ nonsingular is equivalent to $\widetilde{T}$ nonsingular. 
Finally, 
\begin{equation}
\begin{aligned}
X_{k}&=L\widetilde{T}^{-1}L^*
=U_{k}\begin{bmatrix}0&I\\I&\mathcal {P}U_{k-1}\\\end{bmatrix}\widetilde{T}^{-1}
\begin{bmatrix}0&I\\I&U_{k-1}^*\mathcal {P}^*\\\end{bmatrix}U_{k}^*\\
&=U_{k}\left(\begin{bmatrix}-U_{k-1}^*\mathcal {P}^*&I\\I&0\\\end{bmatrix}\widetilde{T}
\begin{bmatrix}-\mathcal {P}U_{k-1}&I\\I&0\\\end{bmatrix}\right)^{-1}U_{k}^* =
U_{k}T_{k}^{-1}U_{k}^* ,
\end{aligned}
\end{equation}
which gives the sought after result.
\end{proof}

\begin{corollary} \label{cor:Xksspd}
Assume that $F\succeq 0$ and $G\succeq 0$. If for some $k>0$,  $X_{k-1}\succeq 0$ , then $X_{k}\succeq 0.$
\end{corollary}

\begin{proof}
The assumption $X_{k-1}\succeq 0$ ensures that $X_{k-1}$ can be written as
$X_{k-1}=U_{k-1}T_{k-1}^{-1}U_{k-1}^*$ with $T_{k-1}\succ 0.$ 
Proposition \ref{prop:lowrank} thus shows that for $G\succeq 0$, $X_k$ can be written as
$X_k = U_k T_k^{-1} U_k^*$ with $T_k$ defined in (\ref{eqn:Tk}). If in addition $F\succeq 0$,
then $T_{k}\succ 0$,  which implies $X_{k}\succeq 0$.
\end{proof}

We conclude this section by showing that the hypothesis that $X_0\succeq 0$ is sufficient for
all subsequent iterates to be well defined.

\begin{proposition}
Suppose $F\succeq 0$ and $G\succeq 0$. Assume that all $\alpha_k$'s have positive real part. 
If $X_0 \succeq 0$,  then all matrices $M_k$, $k =1,2,\ldots$ produced 
by (\ref{eqn:basic_rec})  are nonsingular.
\end{proposition}

\begin{proof}
Theorem \ref{th:existence} states that if 
$X_{k-1}\succeq 0$, then $M_{k}$ is nonsingular and $X_k$ is well defined. 
Corollary \ref{cor:Xksspd} states that if $X_{k-1}\succeq 0$, then $X_{k}\succeq 0$.  
Therefore, choosing $X_0 \succeq 0$,  by induction it follows that
 $M_k$, $k =1,2,\ldots$ produced by (\ref{eqn:basic_rec}) will be nonsingular.
\end{proof}

\section{Considerations on convergence} 
In this section we derive a bound on the angle between the 
approximate and exact invariant subspaces. The result follows classical
strategies to estimate the convergence of subspace iteration, and it
provides a worst case scenario on the actual convergence rate
of the iteration.

We first need to recall some definitions and known relations.  
Let $\sigma(A)$ denote the set of eigenvalues of $A$,
and $D_n(\Ham^*)$ the left c-stable invariant subspace of $\Ham$ (\cite[p.333]{MR1417720}).
Let $\Ham = Q T Q^*$ be the Schur decomposition of $\Ham$, with
\begin{equation}\label{eqn:schur}
Q=\begin{bmatrix}Q_1 & Q_2\end{bmatrix}, \quad T=\begin{bmatrix}T_{11} & T_{12}\\0 &T_{22} \end{bmatrix}, \quad
\sigma(T_{11})\subset \mathbb{C}_-\,,\,\, \sigma(T_{22}) \subset \mathbb{C}_+.
\end{equation}
Then for every $k$ and $\Re(\alpha_k)>0$, 
the Cayley transformation has Schur decomposition $\Sym_k=QT_{(k)}Q^*,$ where 
$T_{(k)}:=\begin{bmatrix}T_{11(k)} & T_{12(k)}\\0 &T_{22(k)} \end{bmatrix}$
with $T_{11(k)}=(T_{11}+\alpha_k I)^{-1}(T_{11}-\bar \alpha_k I) $ having all eigenvalues outside the
unit disk,
while $T_{22(k)}=(T_{22}+\alpha_k I)^{-1}(T_{22}-\bar \alpha_k I)$ has all eigenvalues in the unit disk.

Given two subspaces  ${\mathfrak S}_1$ and ${\mathfrak S}_2$ of $\mathbb{C}^n$ of equal dimension,
their distance is given by (see, e.g., \cite[p.76]{MR1417720}) 
$$
\dist({\mathfrak S}_1,{\mathfrak S}_2)=\|P_1-P_2\|_2,
$$
where $P_i$ is the orthogonal projection matrix onto ${\mathfrak S}_i$.
Finally, (see, e.g., \cite[p.325]{MR1417720})
$$
{\rm sep}(T_{11},T_{22}):=\min_{X\neq 0} \frac{\|T_{11}X-XT_{22}\|_F}{\|X\|_F} .
$$
We are ready to give the main result of this section, whose proof is postponed to the appendix.

\begin{theorem}\label{th:convth}
Let
$\begin{bmatrix}I\\X_0\end{bmatrix}=U_0R_0$ be the skinny QR decomposition of $[I;X_0]$, 
and assume that
$X_0$ is such that 
$$
d=\dist\left(D_n(\Ham^*),\Range\left(\begin{bmatrix}I\\X_0\end{bmatrix}\right)\right)<1 .
$$
If for any $k>0$, the matrix $M_k$ in the iteration (\ref{eqn:basic_rec}) is nonsingular,
then the associated iterate $X_k$ satisfies
\begin{equation}
\begin{aligned}
\dist\left(\Range(\begin{bmatrix}I\\X_+\end{bmatrix}),\Range(\begin{bmatrix}I\\X_k\end{bmatrix})\right)
\leq \gamma 
\left\|\prod_{i=k}^1 T_{22(i)}\right\|_2\left\|\prod_{i=1}^k T_{11(i)}^{-1}\right\|_2
\end{aligned}
\end{equation}
where $\gamma=\frac{\|R_0^{-1}\|_2}{\sqrt{1-d^2}}\left( 1+\frac{\|T_{12}\|_F}{{\rm sep}(T_{11},T_{22})}\right)$.
\end{theorem}
%
%

\vskip 0.1in

Theorem \ref{th:convth} shows that the distance between the exact and approximate
subspaces is bounded in terms of the norms
of the products of the $T_{22(i)}$'s and $T_{11(i)}^{-1}$'s.
From their definition, it holds that
 $\rho(T_{22(i)})<1$  and $\rho(T_{11(i)}^{-1})<1$ for $i=1,\ldots,k$, where
$\rho(T)$ denotes the spectral radius of a square matrix $T$.
Therefore, both norms 
 $\left\|\prod_{i=k}^1 T_{22(i)}\right\|_2$ and 
$\left\|\prod_{i=1}^k T_{11(i)}^{-1}\right\|_2$ tend to zero as $k\to \infty$, thus
ensuring convergence of the iteration. At the same time, the bound shows that
the rate of convergence will depend on the distance of the eigenvalues from
the unit circle. The parameters have the role of optimizing somehow this
distance (cf. section \ref{sec:param}).

Theorem \ref{th:convth} also requires a condition on the initial approximation $X_0$.
A very simple choice of $X_0$, the zero matrix, turns out to satisfy
such hypothesis.

\begin{proposition}
Assume that $A$ is stable and $F,G\succeq 0$.
If  $X_0=0$, then 
${\rm dist}\left(D_n(\Ham^*),{\rm Range}\left(\begin{bmatrix}I\\X_0\end{bmatrix}\right)\right)<1.$
\end{proposition}

\begin{proof}
We have that $(A^*,G)$ is stabilizable 
(see, e.g., \cite[p.12]{Binietal.book.12}), that is there exists $Z_+\succeq 0$ such that
$F+AZ+ZA^*-ZGZ=0$, with $\sigma(A^*-GZ_+)\subset \CC^-$ and
\begin{equation}
\begin{aligned}
&\begin{bmatrix}A^*&-G\\-F&-A\end{bmatrix}
\begin{bmatrix}I\\Z_+\end{bmatrix}=\begin{bmatrix}I\\Z_+\end{bmatrix}(A^*-GZ_+) .
\end{aligned}
\end{equation}
Therefore, $D_n(\Ham^*)=\range\left(\begin{bmatrix}I\\Z_+\end{bmatrix}\right)=
\range\left (\begin{bmatrix}I\\Z_+\end{bmatrix}(I+Z_+^* Z_+)^{-\frac{1}{2}}\right)$, 
where the last matrix in parentheses has orthonormal columns.
Thus,
\begin{eqnarray*}
\dist\left(D_n(\Ham^*),\Range\left(\begin{bmatrix}I\\0\end{bmatrix}\right) \right)
&=&\left\|\begin{bmatrix}0&I\end{bmatrix}\begin{bmatrix}I\\Z_+\end{bmatrix}
(I+Z_+^* Z_+)^{-\frac{1}{2}}\right\|_2 \\
&=&\left\|Z_+(I+Z_+^* Z_+)^{-\frac{1}{2}} \right\|_2<1,
\end{eqnarray*}
where the strict inequality follows from the fact that the (1,1) block of the 
orthonormal basis is nonsingular.
\end{proof}

\section{Subspace iteration for large scale data}

Whenever the problem dimension is very large, the approximate solution matrix 
as expressed in (\ref{eqn:Xnew}) cannot
be explicitly stored. However, if both $F$ and $G$ are low rank, it is possible
to derive a correspondingly low rank factorization of $X_k$ which can be
handled more cheaply.
Proposition \ref{prop:lowrank} exactly shows how to obtain such a form for $X_k$,
and how to update the approximation by increasing the rank at each iteration.
Assuming $G=C^*C$ is low rank, the resulting recursion is given in Algorithm \ref{alg:1}.
We stress that any initial approximation $X_0$ written as $X_0=U_0T_0^{-1}U_0^*$
can be used. Moreover, we notice that the algorithm will not break down if $\alpha_k$ is
an eigenvalue of $\Ham$, as long as $-A^*+\alpha_k I$ is nonsingular, 
the latter being the
only hypothesis required in practice.

\begin{algorithm}[tbh]
    \caption{LRSI: Low-rank Subspace Iteration. Generic implementation. \label{alg:1}}

    \begin{algorithmic}[1]

\STATE INPUT: Given $U_0, T_0$ such that $X_0=U_0T_0^{-1}U_0^*$, and $\alpha_k$, $k=1,2,\ldots$, with $\aa_k=\Re(\alpha_k)$
\FOR{$k=1,2,3, \ldots $}
    \STATE $U_k:=\begin{bmatrix}(-A^*+\alpha_k I)^{-1}(-A^*-\bar\alpha_k I)U_{k-1},&-2\aa_k(-A^*+\alpha_k I)^{-1}C^*\end{bmatrix}$
    \STATE
    $T_k:=\begin{bmatrix}T_{k-1}&0\\0&2\aa_k I\end{bmatrix}
    +2\aa_k\begin{bmatrix}U_{k-1}^*\\C\end{bmatrix} (-A+\bar \alpha_k I)^{-1}F(-A^*+\alpha_k I)^{-1}\begin{bmatrix}U_{k-1}&C^*\end{bmatrix}$
\ENDFOR

 \STATE OUTPUT: $U_k, T_k$ such that $X_k=U_kT_k^{-1}U_k^* \approx X_+$

    \end{algorithmic}
\end{algorithm}

A more effective low rank recursion is obtained by noticing that the term
$$
\begin{bmatrix}U_{k-1}^*\\C\end{bmatrix} (-A+\bar\alpha_k I)^{-1}F(-A^*+\alpha_k I)^{-1}\begin{bmatrix}U_{k-1}&C^*\end{bmatrix}
$$
can be computed without explicitly computing the $n\times n$ inner matrix. 
This operation is particularly
cheap if $F=BB^*$ with $B$ having low column rank.
A closer look at the recurrence matrix 
\begin{eqnarray}\label{eqn:Uk}
U_k:=\begin{bmatrix}(-A^*+\alpha_k I)^{-1}
(-A^*-\bar \alpha_k I)U_{k-1},&-2\aa_k(-A^*+\alpha_k I)^{-1}C^*\end{bmatrix}
\end{eqnarray}
reveals that, except for an innocuous scaling factor,
this is precisely the same iteration matrix obtained when using LR-ADI 
\cite[formulas (4.6)-(4.7)]{Li2002},\cite{MR1742324}. In particular,
when the nonlinear term is zero ($F=0$), the recurrence in Algorithm \ref{alg:1}
corresponds to the LR-ADI iteration. 
As a consequence, we obtain that
$$
{\rm Range}(U_k) = {\rm Range}([(-A^* + \alpha_1 I)^{-1}C^*, \ldots, (-A^* + \alpha_k I)^{-1}C^*]),
$$
namely the generated space is the rational Krylov subspace with poles $\alpha_1, \ldots, \alpha_k$,
$k\ge 1$ \cite[Proposition 7.3]{Li2002}. 
From $X_k = U_k T_k^{-1} U_k^*$ it thus follows that a different basis for the Rational
Krylov subspace could be selected to equivalently 
define $X_k$. More precisely, letting $Q_k$ be any nonsingular matrix
of size equal to the number of columns of $U_k$, then the columns of $U_kQ_k$ are still
a basis for the space, and $X_k = (U_k Q_k) (Q_k^{-1}T_k^{-1}Q_k^{-*}) (U_kQ_k)^*$.
This property is particularly important, as
the matrices $U_k$ in (\ref{eqn:Uk})
 are not efficiently nested: the number of system
solves per iteration increases with the number of iterations.
In the next section we derive a more economical low rank recurrence.

\subsection{An incremental low rank recursion}
We next express the approximate solution in terms
 of a {nested} basis spanning the rational Krylov subspace,
which only requires one system solve with $-A^*+\alpha_k I$ at iteration $k$ 
to expand the space. This is based on the observation that the given basis
is nested for $X_0=0$.

To simplify the presentation here and in the following we shall work
with the corresponding rational function scalar bases. When employing matrices, the
symbol $\lambda$ should be replaced by $-A^*$, while the matrix $C^*$ should end each 
term:  $1/(\lambda+\alpha_k)$ should therefore read $(-A^*+\alpha_k I)^{-1} C^*$.
With this notation, we are going to employ the following basis:
\begin{equation}\label{eqn:nestedV}
\begin{aligned}
V_k:=\begin{bmatrix}\frac{-2\aa_1}{\lambda+\alpha_1 },&\frac{-2\aa_2}{\lambda+\alpha_2 }\frac{\lambda-\bar\alpha_1}{\lambda+\alpha_1},&\ldots &\frac{-2\aa_{k-1}}{\lambda+\alpha_{k-1} }
\displaystyle\prod_{i=1}^{k-2}\frac{\lambda-\bar\alpha_i}{\lambda+\alpha_i},&
\frac{-2\aa_k}{\lambda+\alpha_k }
\displaystyle\prod_{i=1}^{k-1}\frac{\lambda-\bar\alpha_i}{\lambda+\alpha_i}
\end{bmatrix}.\\
\end{aligned}
\end{equation}

\begin{lemma}\label{lemma:Xk-1}
Let $V_k$ be the matrix associated with (\ref{eqn:nestedV}).
If, for some $k>0$, it holds that
$X_{k-1}=V_{k-1}T_{k-1}^{-1}V_{k-1}^*$,
then there exists a nonsingular matrix $Q_k$ such that
$\begin{bmatrix}
-2\aa_k(-A^*+\alpha_k I)^{-1}V_{k-1},&-2\aa_k(-A^*+\alpha_k I)^{-1}C^*
\end{bmatrix}Q_k =V_k$, and a nonsingular $T_k$ such that $X_k=V_kT_k^{-1}V_k^*$.
\end{lemma}

\begin{proof}
From the recursion we get $X_k=\widetilde{U}_k\widetilde{T}_k^{-1}\widetilde{U}_k^*$, where
\begin{equation}
\begin{aligned}
\widetilde{U}_k&:=\begin{bmatrix}(-A^*+\alpha_k I)^{-1}(-A^*-\bar\alpha_k I)V_{k-1},&
-2\aa_k(-A^*+\alpha_k I)^{-1}C^*\end{bmatrix}\\
&=\begin{bmatrix}-2\aa_k(-A^*+\alpha_k I)^{-1}V_{k-1},&
-2\aa_k(-A^*+\alpha_k I)^{-1}C^*\end{bmatrix}+
\begin{bmatrix}V_{k-1},&0\end{bmatrix} , \\
\widetilde{T}_k&:=\begin{bmatrix}T_{k-1}&0\\0&2\aa_k I\end{bmatrix}
    +2\aa_k\begin{bmatrix}V_{k-1}^*\\C\end{bmatrix} (-A+\bar\alpha_k I)^{-1}F(-A^*+\alpha_k I)^{-1}\begin{bmatrix}V_{k-1}&C^*\end{bmatrix}\\
&=\begin{bmatrix}T_{k-1}&0\\0&2\aa_k I\end{bmatrix}
+\frac{1}{2\aa_k}Z F Z^*,
\end{aligned}
\end{equation}
where $Z^* = 
\begin{bmatrix}-2\aa_k(-A^*+\alpha_k I)^{-1}V_{k-1},&-2\aa_k(-A^*+\alpha_k I)^{-1}C^*\end{bmatrix}$.
Since $\Range(Z^*)$ = $\Range(V_k)$, 
there exists $Q_k$ such that
$Z^* = V_k Q_k^{-1}$.
Therefore, 
$\widetilde{U}_k=V_kP_k$, where $P_k:=Q_k^{-1}+\begin{bmatrix}I&\\&0\end{bmatrix}$.

Then $X_k=\widetilde{U}_k\widetilde{T}_k^{-1}\widetilde{U}_k^*=V_k P_k\widetilde{T}_k^{-1}P_k^*V_k^*=V_k(P_k^{-*}\widetilde{T}_kP_k^{-1})^{-1}V_k^*.$
After defining  
\begin{eqnarray}\label{eqn:T}
T_k:=P_k^{-*}\widetilde{T}_kP_k^{-1} , 
\end{eqnarray}
we obtain $X_k=V_kT_k^{-1}V_k^*$.
\end{proof}

Lemma \ref{lemma:Xk-1} shows that
if we can find $Q_k$ explicitly, then we can
update  $X_k=V_kT_k^{-1}V_k^*$ from $X_{k-1}=V_{k-1}T_{k-1}^{-1}V_{k-1}^*$.

In the following, we shall make repeated use of the following simple relation,
which holds for any (not necessarily distinct)  $\alpha_i$, $\alpha_j$:
\begin{equation*}
\begin{aligned}
\frac{\alpha_i-\alpha_j}{(\lambda+\alpha_i)(\lambda+\alpha_j)}+\frac{1}{\lambda+\alpha_i}=\frac{1}{\lambda+\alpha_j}.
\end{aligned}
\end{equation*}

\begin{proposition}
Assume the hypotheses and notation of Lemma \ref{lemma:Xk-1} hold. Then
for any $k>0$, $X_k=V_k T_k^{-1} V_k^*$ with $T_k$ as defined in (\ref{eqn:T}) and
\begin{eqnarray*}
V_k&=&[-2\aa_1 (-A^*+\alpha_1 I)^{-1} C^*, -2\aa_2 (-A^* +\alpha_2I)^{-1}\varphi_1(A) C^*,
\\
& & 
\quad \cdots, 
\quad -2\aa_k (-A^* +\alpha_kI)^{-1} \prod_{i=1}^{k-1}\varphi_i(A)C^*],
\end{eqnarray*}
with $\varphi_i(\lambda)=(\lambda-\bar\alpha_i)/(\lambda+\alpha_i)$.
\end{proposition}

\begin{proof}
We prove the assertion for $C^*$ having one column. 
For more columns, the same result can be written by
expanding the matrices $Q_k$ and $P_k$ defined below
using Kronecker products (see Algorithm \ref{algo:CFADI}).
Let
\begin{equation*}
\begin{aligned}
V_k:=\begin{bmatrix}\frac{-2\aa_1}{\lambda+\alpha_1 },&\frac{-2\aa_2}{\lambda+\alpha_2 }
\frac{\lambda-\bar\alpha_1}{\lambda+\alpha_1},&\ldots 
& \frac{-2\aa_k}{\lambda+\alpha_k }\displaystyle%
\prod_{i=1}^{k-1}\frac{\lambda-\bar\alpha_i}{\lambda+\alpha_i}
\end{bmatrix} , 
\end{aligned}
\end{equation*}
and
\begin{equation}\label{eqn:UU}
\begin{aligned}
\widetilde{U}_k=\begin{bmatrix}(1+\frac{-2\aa_k}{\lambda+\alpha_k})V_{k-1},&\frac{-2\aa_k}{\lambda+\alpha_k }\end{bmatrix} , \quad
\widehat{U}_k:=\begin{bmatrix}\frac{-2\aa_k}{\lambda+\alpha_k}V_{k-1},&\frac{-2\aa_k}{\lambda+\alpha_k }\end{bmatrix} .
\end{aligned}
\end{equation}
We need  to find $Q_k$ such that $\widehat{U}_kQ_k= V_k$. 
Then $\widetilde{U}_k=V_kP_k$, where $P_k=Q_k^{-1}+{\rm blkdiag}(I,0).$
Let
$\widehat{Q}=\begin{bmatrix}
&I\\
1&\\
\end{bmatrix},$ then

\begin{equation*}
\begin{aligned}
\mathbf{U}_{k} 
:=\widehat{U}_k\widehat{Q} 
=\begin{bmatrix}\frac{-2\aa_k}{\lambda+\alpha_k },&\frac{-2\aa_k}{\lambda+\alpha_k}\frac{-2\aa_1}{\lambda+\alpha_1 }, 
&\frac{-2\aa_k}{\lambda+\alpha_k}\frac{-2\aa_2}{\lambda+\alpha_2 }\frac{\lambda-\bar\alpha_1}{\lambda+\alpha_1},
\ldots, \,
\frac{-2\aa_k}{\lambda+\alpha_k}\frac{-2\aa_{k-1}}{\lambda+\alpha_{k-1} }\displaystyle%
\prod_{i=1}^{k-2}\frac{\lambda-\bar\alpha_i}{\lambda+\alpha_i} 
\end{bmatrix} .\\
\end{aligned}
\end{equation*}
Since for any $s=1,2,\ldots,k-1$,
$
\begin{bmatrix}
\frac{-2\aa_s}{\lambda+\alpha_s }
&1
\end{bmatrix}
\begin{bmatrix}
1&1\\
&1
\end{bmatrix}
=\begin{bmatrix}
\frac{\lambda-\bar\alpha_s}{\lambda+\alpha_s }
&1
\end{bmatrix}$ , 
it holds that $\mathbf{U}_{k}\mathbf{Q}=\underline{U}_k$, where
\begin{equation*}
\begin{aligned}
\mathbf{Q}&:=
\begin{bmatrix}
1&1&\\
&1&\\
&&I
\end{bmatrix}
\begin{bmatrix}
1&&&\\
&1&1&\\
&&1&\\
&&&I\\
\end{bmatrix}
\ldots
\begin{bmatrix}
I&&&\\
&1&1&\\
&&1&\\
&&&1\\
\end{bmatrix}
\begin{bmatrix}
I&&\\
&1&1\\
&&1\\
\end{bmatrix},\quad {\rm and}\\
\underline{U}_k&:=
\begin{bmatrix}
\frac{-2\aa_k}{\lambda+\alpha_k },
&\frac{-2\aa_k}{\lambda+\alpha_k}\frac{\lambda-\bar\alpha_1}{\lambda+\alpha_1 },
&\frac{-2\aa_k}{\lambda+\alpha_k}\frac{\lambda-\bar\alpha_2}{\lambda+\alpha_2 }\frac{\lambda-\bar\alpha_1}{\lambda+\alpha_1},
&\ldots,
&\frac{-2\aa_k}{\lambda+\alpha_k}\displaystyle
\prod_{i=1}^{k-1}\frac{\lambda-\bar\alpha_i}{\lambda+\alpha_i}
\end{bmatrix}.
\end{aligned}
\end{equation*}
Since for any $s=1,2,\ldots,k-1$,
\begin{equation*}
\begin{aligned}
\begin{bmatrix}
\frac{-2\aa_k}{\lambda+\alpha_k }&\frac{-2\aa_k}{\lambda+\alpha_k }\frac{\lambda-\bar\alpha_s}{\lambda+\alpha_s }
\end{bmatrix}
\begin{bmatrix}
\frac{\bar\alpha_s+\alpha_k}{2\aa_k}&\\
\frac{\alpha_k-\alpha_s}{-2\aa_k}&1
\end{bmatrix}
=\begin{bmatrix}
\frac{-2\aa_s}{\lambda+\alpha_s }&\frac{-2\aa_k}{\lambda+\alpha_k }\frac{\lambda-\bar\alpha_s}{\lambda+\alpha_s }
\end{bmatrix} , 
\end{aligned}
\end{equation*}
it holds that $\underline{U}_k\underline{Q}=V_k$ where

\begin{equation*}
\begin{aligned}
\underline{Q}&:=
\begin{bmatrix}
\frac{\bar\alpha_{1}+\alpha_k}{2\aa_k}&&&&&\\
\frac{\alpha_k-\alpha_{1}}{-2\aa_k}&\frac{\bar\alpha_{2}+\alpha_k}{2\aa_k}&&&&\\
&\frac{\alpha_k-\alpha_{2}}{-2\aa_k}&&\ddots&&\\
&&&\ddots&\frac{\bar\alpha_{k-1}+\alpha_{k}}{2\aa_k}&\\
&&&&\frac{\alpha_{k-1}-\alpha_{k}}{-2\aa_k}&1\\
\end{bmatrix}.\\
\end{aligned}
\end{equation*}

This implies that we can  determine $Q_k$ and $P_k$ such that
 $\widehat{U}_kQ_k= V_k$ and $\widetilde{U}_k=V_kP_k$, that is
\begin{eqnarray*}
Q_k&=&\widehat{Q}\mathbf{Q}\underline{Q} 
=\begin{bmatrix}
&I\\
1&\\
\end{bmatrix}
\begin{bmatrix}
1&\cdots&\cdots&1\\
&\ddots&\cdots&\vdots\\
&&\ddots&\vdots\\
&&&1\\
\end{bmatrix}
\begin{bmatrix}
\frac{\bar\alpha_{1}+\alpha_k}{2\aa_k}&&&&&\\
\frac{\alpha_1-\alpha_{k}}{2\aa_k}&\frac{\bar\alpha_{2}+\alpha_k}{2\aa_k}&&&&\\
&\frac{\alpha_2-\alpha_{k}}{2\aa_k}&&\ddots&&\\
&&&\ddots&\frac{\bar\alpha_{k-1}+\alpha_{k}}{2\aa_k}&\\
&&&&\frac{\alpha_{k-1}-\alpha_{k}}{2\aa_k}&\frac{\bar\alpha_{k}+\alpha_{k}}{2\aa_k}\\
\end{bmatrix}\\
P_k&=&Q_k^{-1}+
\begin{bmatrix}
I&\\
&0\\
\end{bmatrix}.
\end{eqnarray*}

\end{proof}

We summarize the resulting method in Algorithm \ref{algo:CFADI}.
We remark that this implementation generates a CF-ADI-like basis \cite[Algorithm 2]{Li2002} (the
algorithm will be the same for $F=0$); More precisely, the CF-ADI algorithm uses 
\begin{equation}
\begin{aligned}
&\widetilde{V}_k:=\begin{bmatrix}\frac{\sqrt{2\rea_1}}{\lambda+\alpha_1 },&\frac{\sqrt{2\rea_2}}{\lambda+\alpha_2 }\frac{\lambda-\overline{\alpha_1}}{\lambda+\alpha_1},&\ldots &\frac{\sqrt{2\rea_{k-1}}}{\lambda+\alpha_{k-1} }
\displaystyle\prod_{i=1}^{k-2}\frac{\lambda-\overline{\alpha_i}}{\lambda+\alpha_i},&
\frac{\sqrt{2\rea_{k}}}{\lambda+\alpha_k }
\displaystyle\prod_{i=1}^{k-1}\frac{\lambda-\overline{\alpha_i}}{\lambda+\alpha_i}
\end{bmatrix}\\
&\widetilde{v}_1=\sqrt{2\rea_1}(-A^*+\alpha_1 I)^{-1}C^*\\
&\widetilde{v}_{i+1}=\frac{\sqrt{2\rea_{i+1}}}{\sqrt{2\rea_i}}[I-(\alpha_{i+1}+\overline{\alpha}_i)(-A^*+\alpha_{i+1}I)^{-1}]\widetilde{v}_i
\end{aligned}
\end{equation}
for which a corresponding expression of $X_k$ can be derived.
{The given selection of $v_1, T_1$ makes Algorithm \ref{algo:CFADI} mathematically
equivalent to the recurrence in (\ref{eqn:basic_rec}) with $X_0=0$.}

\begin{algorithm}[tbh]
    \caption{\name: Incremental low rank Subspace Iteration algorithm. 
\label{algo:CFADI}}

    \begin{algorithmic}[1]
  \STATE INPUT $A\in \RR^{n\times n}$,
$C\in\RR^{p\times n}$, $B\in\RR^{n\times q}$,
 $\alpha_k$, $k=1, 2, \ldots$, with $\aa_k=\Re(\alpha_k)$
  \STATE  $v_1:=-2\aa_1(-A^*+\alpha_1 I)^{-1}C^*$,  $V_1:=v_1$,

$T_1:=[2 \aa_1 I+2 \aa_1 C(-A+\bar{\alpha}_1I)^{-1}F(-A^*+\alpha_1 I)^{-1}C^*]$

\FOR{$k=2,3, \ldots $}
    \STATE $v_k:=\frac{\alpha_k}{\alpha_{k-1}}(v_{k-1}-(\alpha_{k-1}+\bar\alpha_{k} )(-A^*+\alpha_k I)^{-1}v_{k-1})$
    \STATE $V_k:=\begin{bmatrix}V_{k-1},&v_k\end{bmatrix}$
    \STATE {\footnotesize$Q_k:=\begin{bmatrix}
            &I\\
            1&\\
            \end{bmatrix}
            \begin{bmatrix}
            1&\cdots&\cdots&1\\
            &\ddots&\cdots&\vdots\\
            &&\ddots&\vdots\\
            &&&1\\
            \end{bmatrix}
            \begin{bmatrix}
            \frac{\bar\alpha_{1}+\alpha_k}{2\aa_k}&&&&&\\
            \frac{\alpha_1-\alpha_{k}}{2\aa_k}&\frac{\bar\alpha_{2}+\alpha_k}{2\aa_k}&&&&\\
            &\frac{\alpha_2-\alpha_{k}}{2\aa_k}&&\ddots&&\\
            &&&\ddots&\frac{\bar\alpha_{k-1}+\alpha_{k}}{2\aa_k}&\\
            &&&&\frac{\alpha_{k-1}-\alpha_{k}}{2\aa_k}&\frac{\bar\alpha_{k}+\alpha_{k}}{2\aa_k}\\
            \end{bmatrix} \otimes I_p $}
    \STATE $P_k:=Q_k^{-1}+
            \begin{bmatrix}
            I&\\
            &0\\
            \end{bmatrix}\otimes I_p$
    \STATE $T_k:=P_k^{-*}\left\{\begin{bmatrix}T_{k-1} &0\\0&2\aa_k I\end{bmatrix}
    +\frac{1}{2\aa_k}Q_k^{-*}V_k^*FV_kQ_k^{-1}\right\}P_k^{-1}$
\ENDFOR

 \STATE OUTPUT: $V_k, T_k$ s.t.  $X_k=V_kT_k^{-1}V_k^*\approx X_+$
    \end{algorithmic}
\end{algorithm}

The algorithm sequentially expands the matrix $V_k$ as the iteration progresses.
If $C^*$ has multiple columns, then the columns of $V_k$ increases correspondingly,
at each iteration. Regardless of the number of columns of $C^*$,
the matrix $V_k$ becomes increasingly ill-conditioned, possibly loosing numerical rank.
Although this fact does not influence the stability of the method,
the whole matrix $V_k$ is required, so that memory requirements expand accordingly. 
However, $V_k$ may be stored as $V_k={\cal V}_k{\cal R}_k$, with ${\cal V}_k$ of (smaller) full
column numerical rank, and the small matrix ${\cal R}_k$ possibly having a larger number of
columns than rows. This way, the much thinner matrix ${\cal V}_k$ can be saved in place
of $V_k$. We do not report the implementation details of this approach, which can
be found in \cite{Lin.PhDthesis.13}, and note that this implementation provides
the same numerical results as the original one, up to the truncation tolerance used.

\begin{remark}
{\rm
Algorithm \ref{algo:CFADI} can be easily generalized to handle 
the following generalized algebraic Riccati equation
\begin{equation*}
G+A^*XE+E^*XA-E^*XFXE=0
\end{equation*}
with $E$ nonsingular.  In particular, from 
$(CE^{-1})^*CE^{-1}+(AE^{-1})^*X+XAE^{-1}-XFX=0$,
it follows that $C^*$ is substituted by $E^{-*}C^*$, and
$(-A^*+\alpha_i I)^{-1}$ by $(-A^*+\alpha_iE^*)^{-1}E^*$.
As a consequence,
only the lines 2 and 4 of Algorithm~\ref{algo:CFADI} require some modifications.
In particular, these two lines are replaced by
\vskip 0.1in
\qquad 2'.  $v_1:=-2\aa_1(-A^*+\alpha_1 E^*)^{-1}C^*$,$V_1:=v_1$,

{
\qquad \quad\, $T_1:=[2\aa_1 I+2\aa_1C(-A+\bar{\alpha}_1 E)^{-1}F(-A^*+\alpha_1 E^*)^{-1}C^*]$
}

\qquad 4'. $v_k:=\frac{\alpha_k}{\alpha_{k-1}}(v_{k-1}-(\alpha_{k-1}+\bar\alpha_{k} )(-A^*+\alpha_k E^*)^{-1}E^*v_{k-1})$
\vskip 0.1in
\noindent
The rest of the algorithm is unchanged. \endproof
}
\end{remark}

\subsection{The shifts selection} \label{sec:param}
Theorem \ref{th:convth} suggests that  if the parameters 
$\{\alpha_i\}$, $i=1,\ldots,k$ are chosen so as to
make the norms of
$\prod_{i=1}^k T_{22(i)}$, $\prod_{i=1}^k T_{11(i)}^{-1}$ small,
then convergence of the subspace iteration will be fast.
  Next proposition gives more insight into the role of the
parameters.

\begin{proposition}\label{prop:chooseshifts}
With the notation of Theorem \ref{th:convth},
assume that the $\alpha_i$'s are such that the matrices
$ T_{22(i)}$, $T_{11(i)}^{-1}$, for all $i=1, \ldots, k$ are well defined.
Then
\begin{equation}
\begin{aligned}
\rho\left(\prod_{i=k}^1 T_{22(i)}\right)=
\rho\left(\prod_{i=1}^k T_{11(i)}^{-1}\right)=
\max_{\lambda \in \lambda_+(\Ham)}
\prod_{i=1}^{k}\left|\frac{\lambda-\overline{\alpha}_i}{\lambda+\alpha_i}\right| .
\end{aligned}
\end{equation}
\end{proposition}

\begin{proof}
From
$T_{22(i)}= (T_{22}+\alpha_i I)^{-1}(T_{22}-\overline{\alpha}_i I)$,
$T_{11(i)}= (T_{11}+\alpha_i I)^{-1}(T_{11}-\overline{\alpha}_i I)$, for
$i=1, \ldots, k$, it follows that
{\small
$$
\sigma\left(\prod_{i=k}^1 T_{22(i)}\right)=
\left\{\prod_{i=1}^{k}\frac{\lambda-\overline{\alpha}_i}{\lambda+\alpha_i}: 
\lambda \in \lambda_+(\Ham)\right\}, \,\,
\sigma\left(\prod_{i=k}^1 T_{11(i)}\right)=
\left\{\prod_{i=1}^{k}\frac{\lambda-\overline{\alpha}_i}{\lambda+\alpha_i}: 
\lambda \in \lambda_-(\Ham)\right\} .
$$
}
From $\lambda_+(\Ham)=-\overline{\lambda_-(\Ham)}$ the result follows.
\end{proof}

Proposition \ref{prop:chooseshifts}, together with the requirement
that all $\alpha_i$ have positive real part, motivate the computation
of the parameters as
\begin{eqnarray}
\{\alpha_1, \ldots, \alpha_k\} &=&
\arg \min_{\alpha_1, \ldots, \alpha_k > 0} 
\max_{\lambda \in \lambda_+(\Ham)}\prod_{i=1}^{k}\left|
\frac{\lambda-\overline{\alpha}_i}{\lambda+\alpha_i}\right| ;
\end{eqnarray} 
{note that here and throughout the paper, we assume that 
the set of parameters is closed under conjugation, that is if $\alpha$ belongs
to the set, also $\bar\alpha$ does. In case of complex data, this constraint
is unnecessary.}

The problem of selecting the parameters is quite similar to the one in ADI 
for the Lyapunov equation (see \cite{Lu1991},\cite{Ellner1991},\cite{Benner.Kuerschner.Saak.12},%
\cite{Penzl2000a}, and the discussion in \cite{Simoncini.survey.13}), 
except that now the maximization
is performed with respect to $\Ham$ instead of $A$.
We implemented a variant of Penzl's algorithm in \cite{Penzl2000a}, which
selects the best $m$ Ritz values of $\Ham$ with positive real part,  among
those obtained in the generated Krylov subspaces with  $\Ham$ and  $\Ham^{-1}$ of
size $m_1$ and $m_2$, respectively. In our simple implementation we did not make any
special effort to preserve the symmetric spectral structure in the computation of the
Ritz values, which should instead be taken into account in case accurate
computation is required.  Our numerical experience indicates that the subspace
iteration strongly depends on the quality of these parameters, and that
different selection strategies than this one may be more effective; see
section~\ref{sec:expes} for further discussion.

\subsection{Computation of the residual norm}\label{sec:res}
Unless the problem size is small, the square residual matrix should not be computed
explicitly. Instead, following similar procedures already used in the
literature (cf., e.g., \cite{MR1742324}), the 
computation of residual {\it norm} can be performed economically, by fully
exploiting the low rank form of the approximate solution.
At iteration $k$, using $X_k=V_kT_k^{-1}V_k^*$ gives

\begin{equation}
\begin{aligned}
&\|A^*X_kE+E^*X_kA-E^*X_kBB^*X_kE+C^*C\|_F\\
&=\left\|\begin{bmatrix}C^*& A^*V_k & E^*V_k\end{bmatrix}
\begin{bmatrix}I&0&0\\ 0&0&T_k^{-1}\\0&T_k^{-1}&T_k^{-1}V_k^*BB^*V_k^{-1}T_k^{-1}\\\end{bmatrix}
\begin{bmatrix}C\\V_k^* A \\ V_k^*E\end{bmatrix} \right\|_F\\
&=\left\|R_k
\begin{bmatrix}I&0&0\\ 0&0&T_k^{-1}\\0&T_k^{-1}&T_k^{-1}V_k^*BB^*V_k^{-1}T_k^{-1}\\\end{bmatrix}
R_k^{*} \right\|_F ,
\end{aligned}
\end{equation}
where $R_k$ is obtained from the
economy-size QR decomposition of $[C^*, A^*V_k,  E^*V_k]$.
Since the basis in $V_k$ is nested, it is possible to {\it update} $R_k$ at
each iteration by means of a Gram-Schmidt type procedure, without
recomputing the decomposition from scratch.

\section{Subspace iteration and Galerkin rational Krylov subspace methods}
In the linear case (i.e., $F=0$), it is known that the ADI method is
tightly connected to the Galerkin 
rational Krylov subspace method (RKSM). More precisely, it was already shown
in \cite{Li2002} that the two approximate solutions stem from the
same type of rational Krylov subspace. 
More recently in \cite{Druskin.Knizhnerman.Simoncini.11} it was proved
that the two methods give exactly the same approximate solution if and
only if
the two spaces use the same shifts, and these shifts coincide
with the mirrored Ritz values of $A$ onto the generated space.

In this section we show that a natural generalization of this
property also holds for our setting, leading to an equivalence
between the subspace iteration and
the Galerkin rational Krylov method applied to the Riccati equation.
We recall here that RKSM determines a solution onto the rational Krylov
subspace by requiring that the residual matrix associated with
the approximate solution $X_k^{(G)}$ be ``orthogonal'' 
to the space; see, e.g.,  \cite{Simoncini.Szyld.Monsalve.13}. More
precisely, setting $R_k^{(G)}:= A^* X_k^{(G)} + X_k^{(G)}A-X_k^{(G)} F X_k^{(G)}+G$,
it holds that $U_k^* R_k^{(G)} U_k = 0$, where the orthonormal columns of $U_k$ span the
rational Krylov subspace. {Writing $X_k^{(G)} = U_k Y_k U_k^*$,
the condition
$0=U_k^* R_k^{(G)} U_k$ correponds to the
reduced equation $U_k^* A^*U_k Y_k + Y_k U_k^* A U_k - Y_k U_k^* F U_k Y_k
+ U_k^* G U_k=0$. This equation admits a unique stabilizable solution $Y_k$ under the
assumptions that $U_k^* A^*U_k$ is stable and $U_k^* F U_k \succeq 0, 
U_k^* G U_k \succeq 0$. Therefore, in this section we assume that $A$ is passive,
that is $(x^*A x )/(x^*x) < 0$ for all $x\ne 0$, so that $U_k^* A^*U_k$ is stable.}

To prove the equivalence, we exploit yet another basis for 
the rational Krylov subspace, namely
\begin{eqnarray}\label{eqn:V_indiv}
V_k = {\range}( [(-A^*+\alpha_1 I)^{-1}C^*, \ldots, (-A^*+\alpha_k I)^{-1}C^*]) ,
\end{eqnarray}
which
appears to be particularly well suited for such a comparison; the same basis was used
to relate ADI and RKSM for the Lyapunov equation
in \cite{Druskin.Knizhnerman.Simoncini.11}. For the basis to be full rank, 
a necessary condition is that all shifts be distinct. In practice, by using the
nested space construction it is readily seen that this condition may be
relaxed by allowing higher negative powers of $(-A^*+\alpha_i I)$ in case $\alpha_i$ is
a multiple shift. {The derivation below could be obtained also for nested
bases \cite{Lin.PhDthesis.13}; we refrain from reporting this approach here
because it is significantly more
cumbersome, without providing better insight.}
We next show how to generate the matrices $Q_k$ and $P_k$ so as to use $V_k$
as reference basis. We assume that $C^*$ has a single column; otherwise,
a Kronecker form as in Algorithm \ref{algo:CFADI} can be used.
%
With the scalar rational function notation we write
\begin{equation}
\begin{aligned}
V_k:=\begin{bmatrix}\frac{1}{\lambda+\alpha_1},&\frac{1}{\lambda+\alpha_2},&\ldots,
&\frac{1}{\lambda+\alpha_{k-1}},\frac{1}{\lambda+\alpha_k}\end{bmatrix} ,\\
\end{aligned}
\end{equation}
together with the definition of $\widetilde{U}_k$ in (\ref{eqn:UU}).
We observe that
\begin{equation*}
\begin{aligned}
\begin{bmatrix}\frac{-2\rea_k}{\lambda+\alpha_k} V_{k-1},&\frac{-2\rea_k}{\lambda+\alpha_k}\end{bmatrix}
=\begin{bmatrix}\frac{-2\rea_k}{\lambda+\alpha_k} \frac{1}{\lambda+\alpha_1},
&\frac{-2\rea_k}{\lambda+\alpha_k}\frac{1}{\lambda+\alpha_2},
&\ldots,
&\frac{-2\rea_k}{\lambda+\alpha_k}\frac{1}{\lambda+\alpha_{k-1}},
&\frac{-2\rea_k}{\lambda+\alpha_k}\end{bmatrix} .
\end{aligned}
\end{equation*}
Moreover,
$$
\begin{bmatrix}\frac{1}{\lambda+\alpha_k}\frac{1}{\lambda+\alpha_s},&\frac{1}{\lambda+\alpha_k}\end{bmatrix}
\begin{bmatrix}\alpha_k-\alpha_s&\\1&1\end{bmatrix}
=\begin{bmatrix}\frac{1}{\lambda+\alpha_s},&\frac{1}{\lambda+\alpha_k}\end{bmatrix}
\,\, \mbox{for}\,\, s=1,\ldots,k-1 ,
$$
and
$$
\begin{bmatrix}\frac{1}{\lambda+\alpha_k}\frac{1}{\lambda+\alpha_s},&\frac{1}{\lambda+\alpha_k}\end{bmatrix}
=\begin{bmatrix}\frac{1}{\lambda+\alpha_s},&\frac{1}{\lambda+\alpha_k}\end{bmatrix}\begin{bmatrix}\frac{1}{\alpha_k-\alpha_s}&\\\frac{1}{\alpha_s-\alpha_k}&1\end{bmatrix}
\,\, \mbox{for}\,\, s=1,\ldots,k-1 .
$$
Therefore, with
\begin{eqnarray}\label{eqn:Qk}
Q_k=
\frac{1}{-2\rea_k}
\begin{bmatrix}
\alpha_k-\alpha_1&&&&\\
&\alpha_k-\alpha_2&&&\\
&&&\alpha_k-\alpha_{k-1}&\\
1&1&\ldots&1&1
\end{bmatrix} ,
\end{eqnarray}
we obtain
\begin{eqnarray}\label{eqn:Pk}
P_k&=&
Q_k^{-1}
+\begin{bmatrix}I&\\&0\end{bmatrix}
=\begin{bmatrix}
\frac{\alpha_1+\bar{\alpha}_k}{\alpha_1-\alpha_k}&&&&\\
&\frac{\alpha_2+\bar{\alpha}_k}{\alpha_2-\alpha_k}&&&\\
&&&\frac{\alpha_{k-1}+\bar{\alpha}_k}{\alpha_{k-1}-\alpha_{k}}&\\
\frac{2\rea_k}{\alpha_k-\alpha_1}&\frac{2\rea_k}{\alpha_k-\alpha_2}&\ldots&\frac{2\rea_k}{\alpha_{k}-\alpha_{k-1}}&-2\rea_k \\
\end{bmatrix} ,
\end{eqnarray}
so that
\begin{equation}\label{eqn:Pkinv}
\begin{aligned}
P_k^{-1}&
=\begin{bmatrix}
\frac{\alpha_1-\alpha_k}{\alpha_1+\bar{\alpha}_k}&&&&\\
&\frac{\alpha_2-\alpha_k}{\alpha_2+\bar{\alpha}_k}&&&\\
&&&\frac{\alpha_{k-1}-\alpha_{k}}{\alpha_{k-1}+\bar{\alpha}_k}&\\
\frac{-1}{\alpha_1+\bar{\alpha}_k}&\frac{-1}{\alpha_2+\bar{\alpha}_k}&\ldots&\frac{-1}{\alpha_{k-1}+\bar{\alpha}_k}&\frac{-1}{2\rea_k} \\ 
\end{bmatrix} .
\end{aligned}
\end{equation}

As already mentioned, the approximation $X_k$ can be written in terms of
the new basis and representation matrix as
$X_k = V_k T_k^{-1} V_k^*$ with $V_k$ as in (\ref{eqn:V_indiv}) and
\begin{eqnarray}\label{eqn:Tk_indiv}
\!\! T_{1}=\frac{1}{2\rea_1}+\frac{V_1^*FV_1}{2\rea_1}, \,\,
T_k=P_k^{-*}\left(\begin{bmatrix}T_{k-1}&0\\0&2\rea_k I\end{bmatrix}
    +\frac{1}{2\rea_k}Q_k^{-*}V_k^*FV_kQ_k^{-1}\right)P_k^{-1},
\end{eqnarray}
 where $P_k^{-1}$  and
$Q_k$ are as defined in (\ref{eqn:Pkinv}) and (\ref{eqn:Qk}), respectively.
Here we focus on the use of this formulation for demonstrating the connection
of our approach with RKSM.
We first show that the reduced matrix $T_k$ satisfies a linear matrix equation.
\begin{proposition}\label{prop:Teqn}
Let $V_k$ and $T_k$ be as in (\ref{eqn:V_indiv}) and (\ref{eqn:Tk_indiv}), respectively,
 and $\ba_k={\rm diag}(\alpha_1, \ldots, \alpha_k)$.
Let $\bone=[1, \ldots, 1]^*$.
Then
\begin{equation}\label{eqn:neweqn}
\ba_k^*T_k+T_k \ba_k-V_k^*FV_k-\bone\bone^*=0 .
\end{equation}
\end{proposition}
\begin{proof}
With $P_k^{-1}$ in (\ref{eqn:Pkinv}) and
$Q_k^{-1}$ expressed via (\ref{eqn:Pk}), we first observe that
%
%
%
\begin{equation*}
Q_k^{-1}P_k^{-1}=\left(P_k-\begin{bmatrix}I&\\&0\end{bmatrix}\right)P_k^{-1}
=
{\rm diag}\left(\frac{2\rea_k}{\alpha_1+\bar{\alpha}_k}, \ldots, \frac{2\rea_k}{\alpha_{k-1}+\bar{\alpha}_k},1\right) .
\end{equation*}
We are going to prove  that
$T_{k}(i,j)=\frac{1+\widetilde{F}_{ij}}{\bar{\alpha_i}+\alpha_j}$, for $i,j \leq k$ by induction on
$k$, where $\widetilde{F}_{ij}=V_k^*FV_k$.
For $k=1$ it can be easily verified that
$T_{1}=\frac{1}{2\rea_1}+\frac{V_1^*FV_1}{2\rea_1}$.
Then assume that the  relation holds for $T_{k-1}$. 
Noticing the structure of $P_k^{-1}$ and $Q_k^{-1}P_k^{-1}$, for $i<k$, $j<k$ we have
\begin{equation*}
\begin{aligned}
T_k(i,j)&=e_i^*P_k^{-*}\begin{bmatrix}T_{k-1}&0\\0&2\rea_k \end{bmatrix}P_k^{-1}e_j
    +\frac{1}{2\rea_k}e_i^*P_k^{-*}Q_k^{-*}V_k^*FV_kQ_k^{-1}P_k^{-1}e_j\\
&=\begin{bmatrix}\frac{\bar{\alpha}_i-\bar{\alpha}_k}{\bar{\alpha}_i+\alpha_k}& \frac{-1}{\bar{\alpha}_i+\alpha_k}\end{bmatrix}
\begin{bmatrix}\frac{1+\widetilde{F}_{ij}}{\bar{\alpha}_i+\alpha_j}&\\& 2\rea_k \end{bmatrix}
\begin{bmatrix}\frac{{\alpha}_j-{\alpha}_k}{{\alpha}_j+\bar{\alpha}_k}\\ \frac{-1}{{\alpha}_j+\bar{\alpha}_k}\end{bmatrix}
+\frac{1}{2\rea_k}\frac{2\rea_k}{\bar{\alpha}_i+\alpha_k}\widetilde{F}_{ij}\frac{2\rea_k}{{\alpha}_j+\bar{\alpha}_k}\\
&=\frac{\bar{\alpha}_i-\bar{\alpha}_k}{\bar{\alpha}_i+\alpha_k}\frac{1+\widetilde{F}_{ij}}{\bar{\alpha}_i+\alpha_j}\frac{{\alpha}_j-{\alpha}_k}{{\alpha}_j+\bar{\alpha}_k}
+\frac{2\rea_k}{(\bar{\alpha}_i+\alpha_k)({\alpha}_j+\bar{\alpha}_k)}
+\frac{2\rea_k\widetilde{F}_{ij}}{(\bar{\alpha}_i+\alpha_k)({\alpha}_j+\bar{\alpha}_k)}\\
&=\frac{1+\widetilde{F}_{ij}}{\bar{\alpha}_i+\alpha_j} .
\end{aligned}
\end{equation*}
For $j=k$ and $i\leq k$ we obtain
\begin{equation*}
\begin{aligned}
T_k(i,k)&=e_i^*T_ke_k=e_i^*P_k^{-*}\begin{bmatrix}T_{k-1}&0\\0&2\rea_k \end{bmatrix}P_k^{-1}e_k
+\frac{1}{2\rea_k}e_i^*P_k^{-*}Q_k^{-*}V_k^*FV_kQ_k^{-1}P_k^{-1}e_k\\
&=e_i^*P_k^{-*}(-e_k)+\frac{1}{2\rea_k}\frac{2\rea_k \widetilde{F}_{ik}}{\alpha_i+\bar{\alpha}_k}
=\frac{1+\widetilde{F}_{ik}}{\bar{\alpha}_i+\alpha_k} .
\end{aligned}
\end{equation*}

The structure of $T_k(k,j)$, $j\le k$ is obtained by symmetry, and the proof is completed.
\end{proof}

We notice that Proposition~\ref{prop:Teqn} also
shows that the principal $(k-1)\times (k-1)$ diagonal
block  of $T_k$ coincides with $T_{k-1}$. 
As an immediate consequence of this fact,
we show that the approximate solution $X_k$ can be updated from $X_{k-1}$
with a rank-one matrix (a rank-$p$ matrix if $C^*$ has $p$ columns); 
this is similar to what one finds with CF-ADI. 
In addition, the approximation sequence is weakly monotonically increasing.

\begin{theorem}
For $k>0$, the approximate solution $X_k$ is such that $X_k - X_{k-1}$ has rank one.
Moreover, for all $k>0$, $X_k \succeq X_{k-1}$.
\end{theorem}

\begin{proof}
Let $X_{k}=V_{k}T_k^{-1}V_{k}^{*}$, with
$V_k$ and $T_k$ as in (\ref{eqn:V_indiv}) and (\ref{eqn:Tk_indiv}), respectively.
Since $T_k(1:k-1,1:k-1)=T_{k-1}$, we have
$X_k-X_{k-1}=V_k\left(T_k^{-1}-\begin{bmatrix}T_{k-1}^{-1}&\\& 0\end{bmatrix}\right)V_k^*$.

We next show that the matrix in parentheses has rank one.
Let $T_{k-1}=LL^*$ be the Cholesky decomposition of $T_{k-1}$. Then
\begin{equation*}
\begin{aligned}
&T_k=\begin{bmatrix}L&0\\l^*&m \end{bmatrix}\begin{bmatrix}L^*&l\\0 &\bar m \end{bmatrix} ,
\quad
T_k^{-1}=\begin{bmatrix}L^{-*}&-L^{-*}l\bar m^{-1}\\0&\bar m^{-1} \end{bmatrix}
\begin{bmatrix}L^{-1}& 0 \\-m^{-1}l^*L^{-1}&m^{-1} \end{bmatrix} .\\
\end{aligned}
\end{equation*}
By explicitly writing down the (1,1) block of $T_k^{-1}$ it follows
\begin{equation*}
\begin{aligned}
T_k^{-1}-\begin{bmatrix}T_{k-1}^{-1}&\\&0 \end{bmatrix}
&=\begin{bmatrix}-L^{-*}l  \bar m^{-1}\\\bar m^{-1} \end{bmatrix}
\begin{bmatrix}-m^{-1}l^* L^{-1}&m^{-1} \end{bmatrix} = : {\pmb \ell} {\pmb \ell}^* ,
\end{aligned}
\end{equation*}
which has rank one, as stated. Finally,
$X_k = X_{k-1} + V_k {\pmb \ell} {\pmb \ell}^* V_k^*$ with
$V_k {\pmb \ell} {\pmb \ell}^* V_k^* \succeq 0$,  thus completing the proof.
\end{proof}

With these results in hand, we are able to show that for $C^*$ having a single column,
the Riccati equation residual associated with $X_k$ is also a rank-one matrix.

\begin{proposition}\label{prop:rankoneres}
Assume $C$ is rank-one. Then the residual matrix
$R_k=C^*C+A^*X_k+X_kA-X_kFX_k$ is also rank-one.
\end{proposition}

\begin{proof}
Recalling the notation leading to (\ref{eqn:neweqn}), we can write
$A^* V_k = -C^* \bone^* + V_k \ba_k$, so that
\begin{equation*}
\begin{bmatrix}
C^*,&A^*V_k,&V_k
\end{bmatrix}
=[C^*,V_k]
\begin{bmatrix}
1&-\bone^*&0\\
0&\ba_k&I\\
\end{bmatrix} .
\end{equation*}
Using (\ref{eqn:neweqn}),
\begin{eqnarray}
R_k&=&[C^*,A^*V_k,V_k]
\begin{bmatrix}
1&0&0\\
0&0&T_k^{-1}\\
0&T_{k}^{-1}&-T_k^{-1}V_k^* F V_k T_k^{-1}\\
\end{bmatrix}
[C^*,A^*V_k,V_k]^*
 \nonumber\\
&=&[C^*,V_k]
\begin{bmatrix}
1&-\bone^*T_k^{-1}\\
-T_{k}^{-1}\bone&T_k^{-1}\ba^*+\ba_kT_k^{-1}-T_k^{-1}V_k^* F V_kT_k^{-1}\\
\end{bmatrix}
[C^*,V_k]^* \nonumber\\
&=&[C^*,V_k]
\begin{bmatrix}
1&-\bone^*T_k^{-1}\\
-T_{k}^{-1}\bone&T_{k}^{-1}\bone\bone^*T_{k}^{-1}\\
\end{bmatrix}
[C^*,V_k]^* \nonumber\\
&=&[C^*,V_k]
\begin{bmatrix}
1\\
-T_{k}^{-1}\bone\\
\end{bmatrix}
\begin{bmatrix}
1&
-\bone^*T_{k}^{-1}\\
\end{bmatrix}
[C^*,V_k]^* , \label{eqn:rankoneRes}
\end{eqnarray}
which is a rank-one matrix.
\end{proof}


We can thus state the main result of this section, which gives necessary
and sufficient conditions for the subspace iteration and RKSM for the
Riccati equation to be mathematically equivalent. 
The equivalence follows from the uniqueness of the Galerkin solution onto
the given space, determined by RKSM, {following from the uniqueness of
the stabilizing solution of the reduced problem}.

\begin{theorem}\label{th:ritz}
{Assume $A$ is passive, and}
assume the notation and assumption of Proposition \ref{prop:rankoneres} hold.
Let $K_k:=(V_k^*V_k)^{-1}V_k^*A^*V_k$, $g_k:=(V_k^*V_k)^{-1}V_k^*C^*$. Then
\begin{enumerate}
\item[$(i)$] The subspace iteration provides a Galerkin method on Range($V_k$), namely $V_k^*R_kV_k=0$,
if and only if $g_k=T_k^{-1} \bone$.
\item[$(ii)$]  $V_k^*R_kV_k=0$ if and only if 
$\ba_k^*T_k+T_kK_k-V_k^*FV_k=0$; in particular,
the poles are the mirrored Ritz values of  $A^*-X_k F$, namely 
$$
\alpha_j = - \bar \lambda_j , \,\, j=1, \ldots, k
$$
where $\lambda_j$ are the properly sorted
eigenvalues of {\footnotesize $(V_k^*V_k)^{-1}V_k^*( A^*-V_kT_k^{-1}V_k^*F)V_k$.}
\end{enumerate}
\end{theorem}

\begin{proof}
Using the relations in the proof of Proposition \ref{prop:rankoneres},
we first notice that the relations
$V_k^*A^*V_k=-V_k^*C^*\bone^*+V_k^*V_k\ba_k$ and
$K_k=-g_k\bone^*+\ba_k$ hold. Then from (\ref{eqn:rankoneRes}) we obtain
\begin{equation*}
\begin{aligned}
V_k^*R_k V_k&=[V_k^*C^*-(V_k^*V_k)T_k^{-1}\bone][V_k^*C^*-(V_k^*V_k)T_k^{-1}\bone]^*\\
&=(V_k^*V_k)[g_k-T_k^{-1}\bone][g_k-T_k^{-1}\bone]^*(V_k^*V_k) ,
\end{aligned}
\end{equation*}
from which the necessary and sufficient condition in $(i)$ follows.

For proving $(ii)$, let us first assume that $V_k^*R_kV_k=0$. Then,
using $T_k g_k= \bone$ in $(i)$ and $K_k=-g_k\bone^*+\ba_k$ we have
from Proposition \ref{prop:Teqn}
\begin{equation}
\begin{aligned}
0&=\ba_k^*T_k+T_k\ba_k-V_k^*FV_k-\bone\bone^*\nonumber\\
&=\ba_k^*T_k+T_k\ba_k-V_k^*FV_k-T_kg_k\bone^*\nonumber\\
&=\ba_k^*T_k+T_k\ba_k-V_k^*FV_k+T_k(K_k-\ba_k)\nonumber\\
&=\ba_k^*T_k+T_kK_k-V_k^*FV_k . \label{eqn:TK}
\end{aligned}
\end{equation}
To prove the opposite direction, we start from $0=\ba_k^*T_k+T_kK_k-V_k^*FV_k$
and go backward to $0=\ba_k^*T_k+T_k\ba_k-V_k^*FV_k-T_k(V_k^*V_k)^{-1}V_k^* C^*\bone^*$. 
Since our iterates satisfy
(\ref{eqn:neweqn}), it must follow that $T_k(V_k^*V_k)^{-1}V_k^* C^*=\bone$,
that is $g_k=T_k^{-1} \bone$.

 Finally, we notice that (\ref{eqn:TK})
 is equivalent to $T_k^{-1}\ba^*T_k=T_k^{-1}V_k^*FV_k-K_k$, and
\begin{eqnarray*}
T_k^{-1}\ba^*T_k&=&(V_k^* V_k)^{-1}V_k^*V_k T_k^{-1}V_k^*FV_k-K_k \\
&=&
(V_k^* V_k)^{-1}V_k^* ( V_k T_k^{-1}V_k^*F - A)V_k ,
\end{eqnarray*}
so that
the eigenvalues of the first and last matrices 
coincide, and the eigenvalues of $(V_k^* V_k)^{-1}V_k^* ( V_k T_k^{-1}V_k^*F - A)V_k$
coincide with those of 
$(V_k^* V_k)^{-1/2}V_k^* ( V_k T_k^{-1}V_k^*F - A)V_k(V_k^* V_k)^{-1/2}$, where
the columns of $V_k(V_k^* V_k)^{-1/2}$ define an orthogonal basis for the space.
 Therefore, these are the Ritz values of
$V_k T_k^{-1}V_k^*F - A$ onto the space Range($V_k$).
%
%
\end{proof}

\begin{remark}
{\rm
The previous theorem provides insight into the estimation of the poles of RKSM,
when a greedy algorithm is used to generate a pole sequence ``on the fly'':
in the linear case, poles are estimated by an optimization strategy of a
scalar rational function on a certain region of the complex plane. The function
has poles at the already computed shifts and
zeros at the Ritz values of $-A$ in the current space \cite{Druskin.Simoncini.11}. 
The results of Theorem \ref{th:ritz}
suggest that in the quadratic case, an alternative choice could be given
by the Ritz values of $-A^*+X_k^{(G)}F$, where $X_k^{(G)}$ is the current 
approximate solution. The very preliminary experiments reported in 
Example~\ref{ex:toeplitz} seem to encourage the use of this strategy when
$A$ is nonnormal and $X_k^{(G)}F$ is sizable in norm. \endproof
}
\end{remark}

In the case when the Ritz values of $A$, $\lambda_j$, are considered, the condition
$\alpha_j = - \bar \lambda_j$ is associated with the optimality of the generated
rational Krylov subspace as a model order reduction process for a linear
dynamical system; see, e.g.,
\cite{Gugercin2008a}. Whether different optimality
results could be shown in our setting remains an open problem.

\section{Numerical experiments}\label{sec:expes}
In this section we report on our numerical experience with the subspace iteration
described in Algorithm \ref{algo:CFADI}. Experiments were performed in Matlab (\cite{matlab7})
with version 7.13 (R2011b) of the software.

We do not report these numerical experiments to propose the method as a
valid competitor of, e.g., rational Krylov subspace solvers, as
the large majority of our experiments showed otherwise. Having the extra
feature of the Galerkin projection, RKSM with the same poles will in general
be superior to subspace iteration, both in terms of number of iterations and
memory requirements. Instead, our purpose is to explore what the expected
performance of the method will be, and highlight the relations with the
Galerkin procedure, specifically in connection with the pole selection.
This analysis also lead us to the derivation of a possibly more effective pole selection
for RKSM, compared with what was used, e.g., in \cite{Simoncini.Szyld.Monsalve.13}.
All experiments are performed with $F$ and $G$ of rank one. Similar results 
may be obtained with matrices of larger rank.
{All plots report the computed residual norm, according to
section~\ref{sec:res}, versus the space dimension. In fact, for \name\ this refers
to the number of columns in the matrix $V_k$  in Algorithm~\ref{algo:CFADI}, 
since the numerical rank of that
matrix may be lower.}

We do not report experimental comparisons with other methods such as
inexact Newton, as they are available in \cite{Simoncini.Szyld.Monsalve.13}, at
least with respect to projection-type methods.

\begin{figure}[htb]
\centering
\includegraphics[width=.48\textwidth]{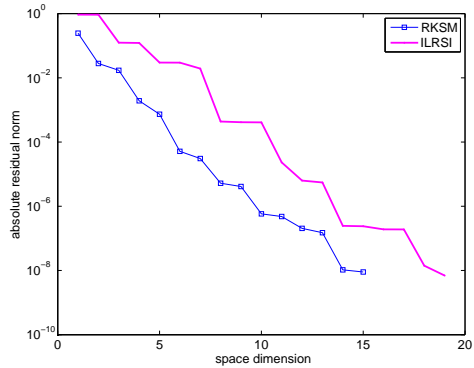}
\caption{Example \ref{ex:1}. Performance of the analyzed methods.\label{fig:ex1}}
\end{figure}

In all our examples with the subspace iteration algorithm \name,
the poles are computed a-priori. Unless 
explicitly stated otherwise,
these are computed using Penzl's algorithm \cite{Penzl2000a} on the matrix $A$ 
{(when used for these specific problems, the variant of 
Penzl's algorithm using $\Ham$ mentioned in section \ref{sec:param}
did not give appreciably better results)}.
In the first two examples, 
the performance of the new method is compared with that of adaptive RKSM, as
used, for instance, in \cite{Simoncini.Szyld.Monsalve.13}, where the poles
are computed adaptively. We notice that the main computational cost per iteration,
namely the solution of the shifted system with $A$, is the same for both
methods, therefore the number of solves may represent a good measure for
the comparison.
%

\begin{example}\label{ex:1}
{\rm 
We consider the (scaled) discretization of the Laplace operator on the unit square, with
100 interior points in each direction, so that the resulting matrix $A$ has dimension
$n=10 000$. The matrices $F$ and $G$ are given as $F=bb^*$ and $G=c^*c$ with
$b={\mathbf 1}$ and $c^*=e_1$. The performance of \name\
is reported in 
Figure~\ref{fig:ex1}, together with that of RKSM. 
The convergence rate is similar for the two methods, although RKSM consistently shows 
smaller residual norm.
}
\end{example}

\begin{figure}[htb]
\centering
\includegraphics[width=.48\textwidth]{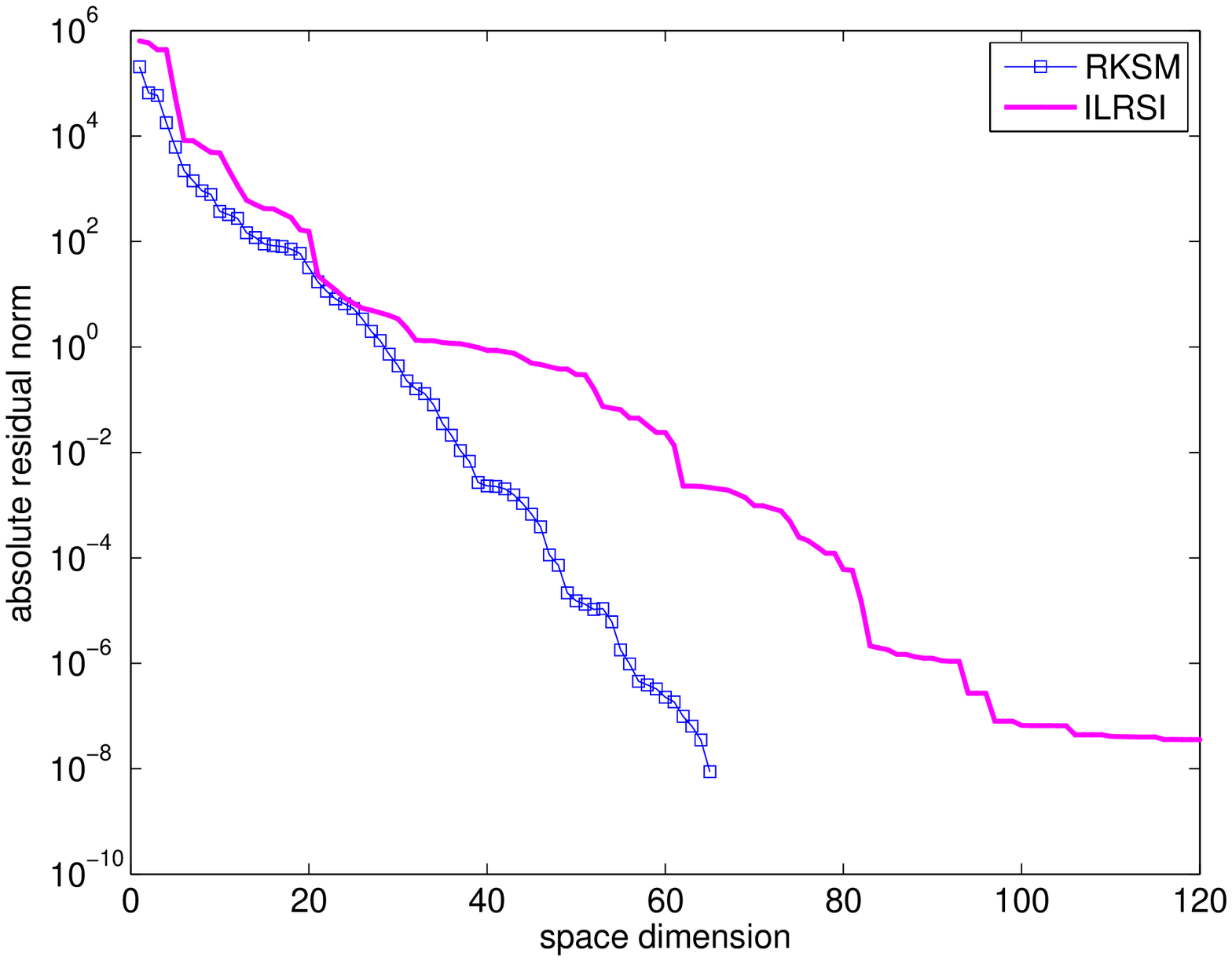}
\includegraphics[width=.48\textwidth]{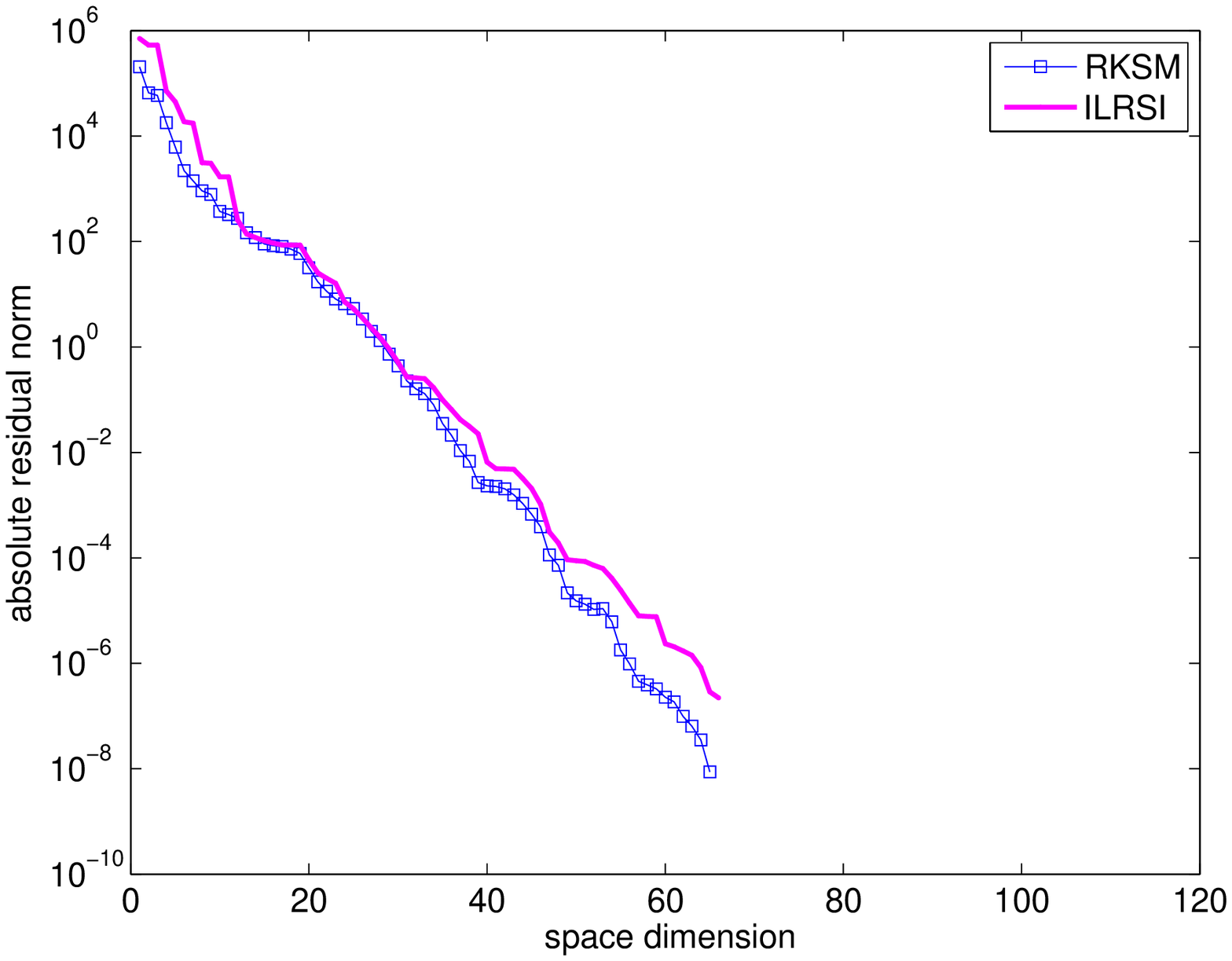}
\caption{Example \ref{ex:2}. Performance of the analyzed methods. For \name, the
parameters are obtained either via {\tt lyapack} (left) or, as
those of RKSM (right).  \label{fig:ex2}}
\end{figure}

\begin{example}\label{ex:2}
{\rm
In this example, we consider the data set FLOW from the Oberwolfach
collection (\cite{Collection2003}), with $n= 9669$; $B$ and $C^*$ have a single column.
The convergence histories of the subspace iteration and of adaptive RKSM
are reported in Figure~\ref{fig:ex2}.  The left plot shows adaptive RKSM
and \name, where for the latter the poles were pre-computed with 
Penzl's algorithm on $A$.  For this example, the adaptive RKSM
is able to obtain an accurate solution appreciably earlier than the new method.
In the right plot, subspace iteration was run with the poles adaptively 
generated by RKSM, showing a convergence history very similar to that of RKSM.
Such different performance confirms what one usually finds in the linear case:
the behavior of ADI is very sensitive to the poles choice.
}
\end{example}

\begin{example}\label{ex:toeplitz}
{\rm
We consider the $500\times 500$ Toeplitz matrix 
$$
A=\begin{bmatrix} 2.5 & 1 & 1 & 1 & 0 & \ddots &\\ -1 & 2.5 & 1 & 1 & 1 & 0 & \ddots  \\
0 & -1 & 2.5 & 1 & 1 & 1 & \ddots  \\ 
\ddots & \ddots & \ddots & \ddots & \ddots & \ddots & \ddots  \\
\ddots & \ddots & \ddots & \ddots & \ddots & \ddots & \ddots  \\
  &  &  &   & 0 & -1 &  2.5  
\end{bmatrix} ,
$$
with $C=[1,-2,1,-2,\ldots]$, while $B={\mathbf 1}$
normalized or non-normalized. This type of matrices  is known to be
very non-normal, which implies that at small perturbations of the entries there
may correspond very large spectral perturbations; see, e.g., \cite[ch.7]{Trefethen.Embree.05}.
Figure \ref{fig:toeplitz} reports the convergence history with adaptive
RKSM and \name, when the latter uses the poles computed
by the former. The left-most plot stems from using $B/\|B\|$ in place of $B$, whereas
the middle plot refers to the unnormalized case. 
While the performance of RKSM only slightly degrades in the unnormalized case,
that of subspace iteration drastically changes, showing almost complete stagnation. Indeed,
two very large in modulus eigenvalues of $\Ham$ are mapped by Cayley's transformation
to an area very close to the unit circle, for all parameters $\alpha_k$, thus causing
very slow convergence.
The right-most plot shows the performance of the methods with $B={\mathbf 1}$ (unnormalized),
when the parameters in RKSM were computed by using the current Ritz values of
$A-BB^*X_k^{(G)}$ instead of those of $A$ (cf. Theorem \ref{th:ritz}). 
We can readily see that performance
of both methods is significantly improved, and in particular no complete
stagnation occurs for subspace iteration. A closer look reveals that for $X$ exact, 
$A-BB^*X$ has an isolate eigenvalue close to $-250$ (apparently
caused by the modification induced by the norm of $B$),
which is not captured by the Ritz values of $A$ alone. When $B$ is normalized,
the Ritz values of $A-BB^*X_k^{(G)}$ do not differ significantly from those of $A$, and
thus performance does not differ much. So in this case where the spectrum of
$A-BB^*X$ differs significantly from that of $A$, using the Ritz values of
$A-BB^*X_k^{(G)}$ for the adaptive computation of the parameters yields significantly
better performance. {This phenomenon deserves further study.}
}
\end{example}

\begin{figure}[htb]
\centering
\includegraphics[width=.32\textwidth]{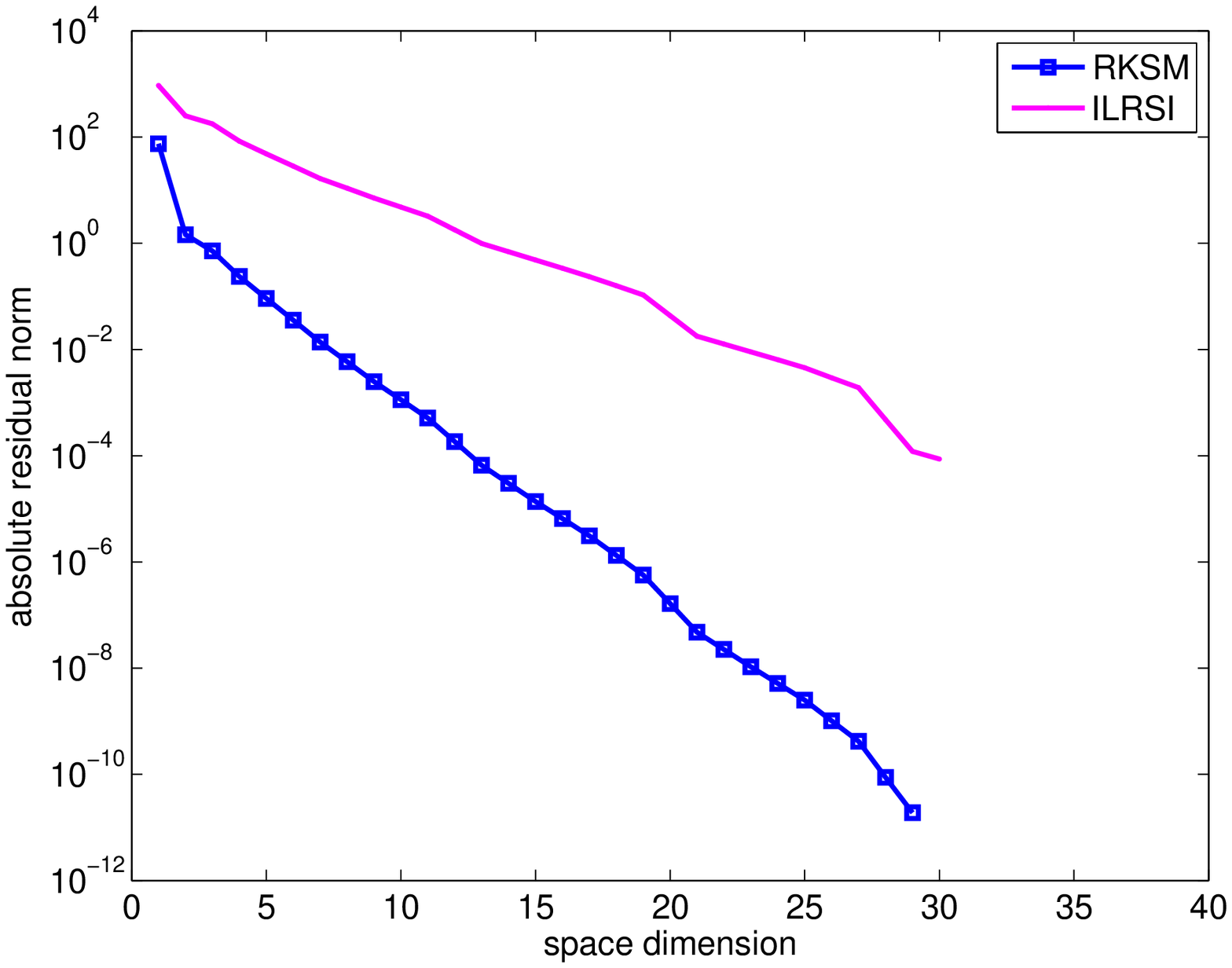}
\includegraphics[width=.32\textwidth]{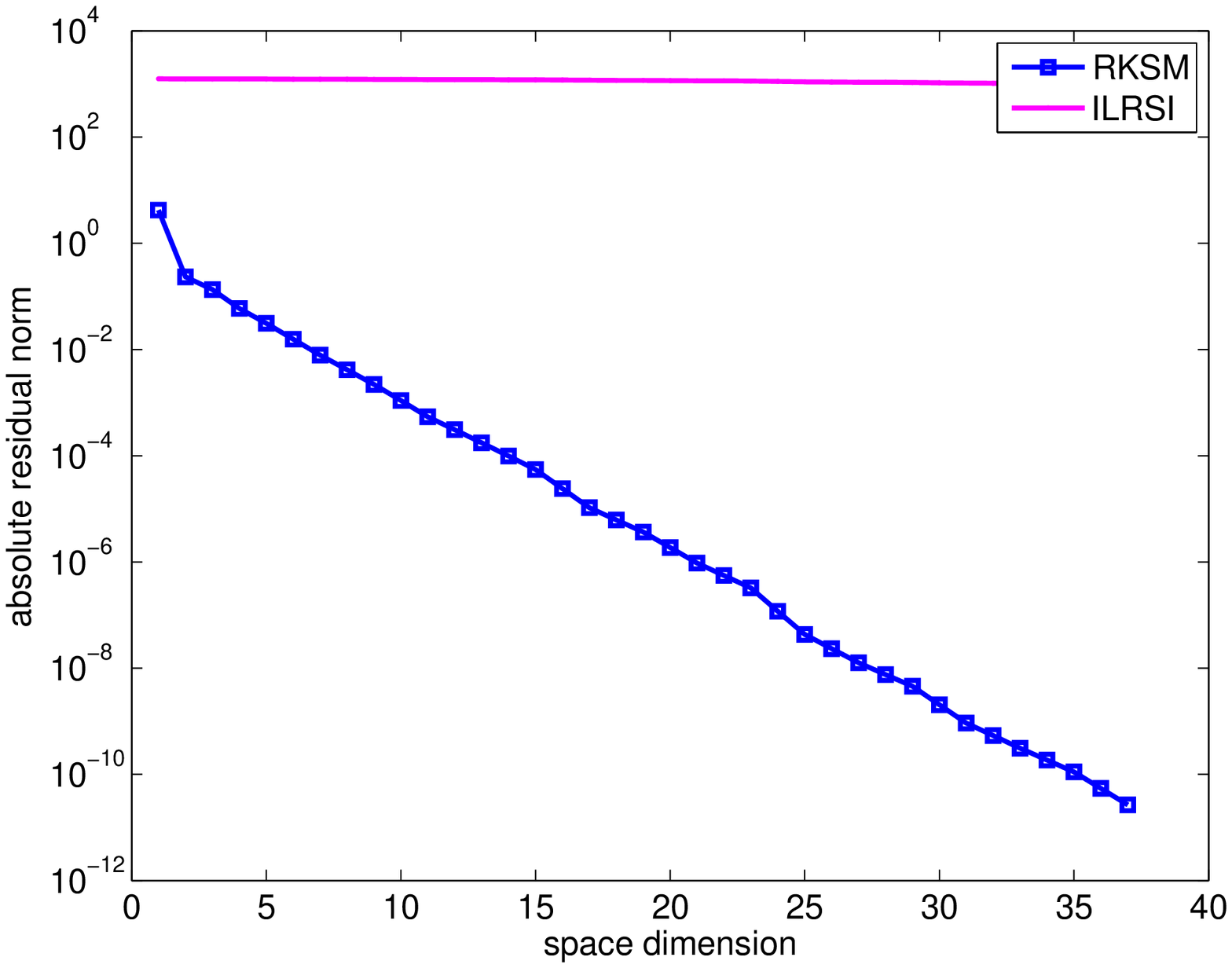}
\includegraphics[width=.32\textwidth]{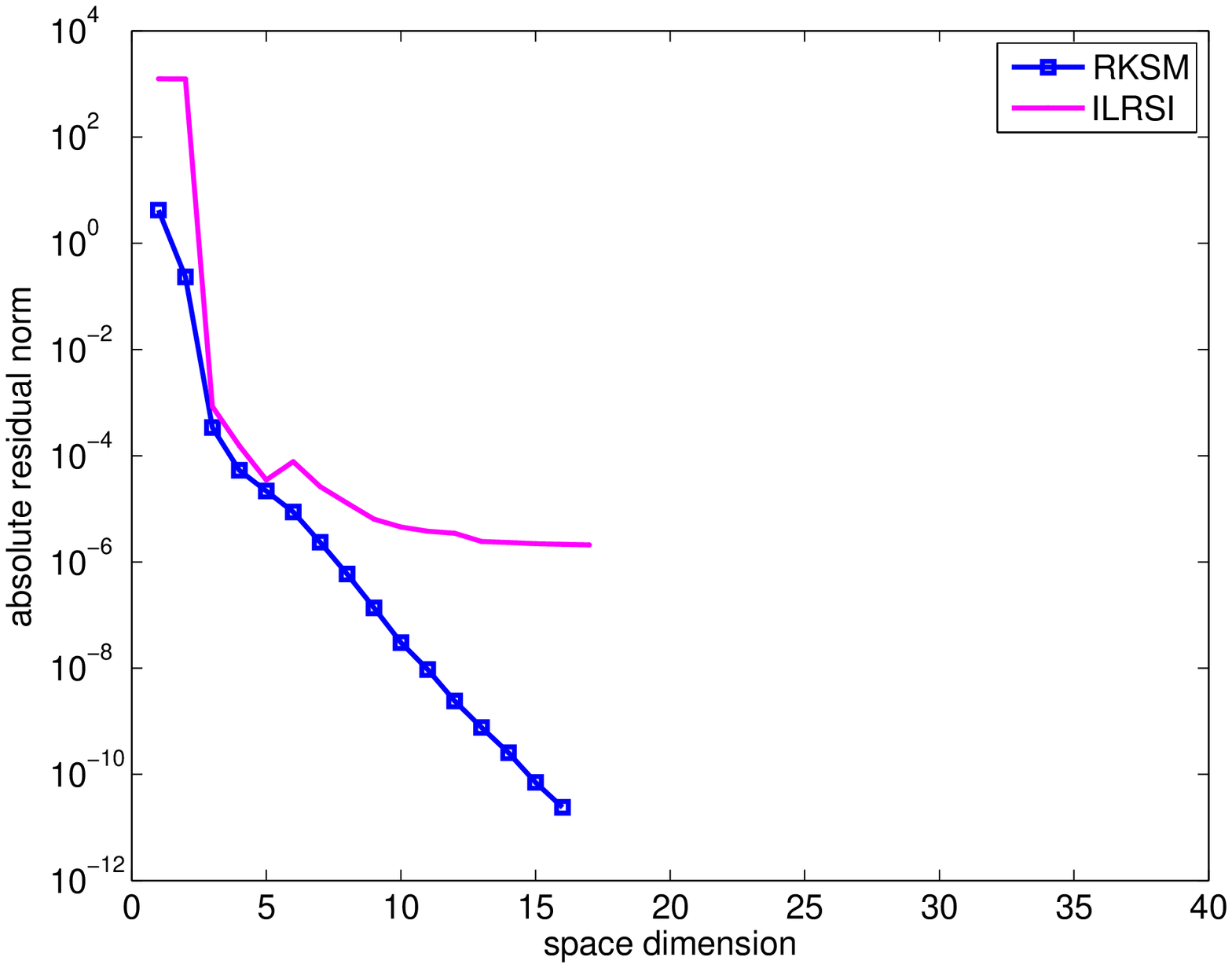}
\caption{Example \ref{ex:toeplitz}: Toeplitz matrix.\label{fig:toeplitz} Subspace iteration uses parameters
computed with adaptive RKSM. Left: $B$ normalized to have unit
Euclidean norm. Middle: unnormalized $B$. Right: parameters computed with ``Stabilized'' Ritz values
in RKSM for unnormalized $B$.}
\end{figure}

\section{Conclusions}
We have derived a computationally feasible subspace iteration algorithm for
the approximation of the solution to the large scale algebraic Riccati equation,
when the matrices $F$ and $G$ have low rank.  {The
new method coincides with the ADI method 
in the linear equation case}. Consequently, the 
performance of the new method depends on certain
parameters, whose selection follows similar reasonings than those
used for ADI. Our derivation also shows that ADI may be viewed
as a subspace iteration method for the Hamiltonian matrix with $F=0$.
Other issues deserve further future analysis, such as the choice of the
initial approximation $X_0$, which, together with a refined 
shift selection, could considerably speed up the process.
{Although we have worked throughout with real data,
the method is also well suited for complex data, as long as the poles
are chosen in a suitable manner.}

We have also derived a new insightful connection of the proposed method
with the Galerkin rational Krylov
subspace scheme, which aims at generalizing known equivalence
in the linear case. Such connection opens up a new venue for
the understanding of the convergence properties of RKSM, which
is a competitive alternative to Newton based approaches. We
plan to explore this problem in future work.

\section*{Acknowledgments}
This work was performed while the first author was visiting
the Department of Mathematics of the Universit\`a di Bologna
{during the period Sept 2011 - Aug 2013,}
supported by
{fellowship {2011631028} from the China Scholarship Council (CSC).}

\section*{Appendix}  In this appendix we prove Theorem \ref{th:convth}, ensuring convergence
of the subspace iteration. The proof is an 
adaptation of the general proof in \cite[Theorem 7.3.1, page 337]{MR1417720}
to our context.

Let ${\cal H} = Q T Q^*$ be the block Schur decomposition of $\cal H$, with $T=[T_{11}, T_{12}; 0, T_{22}]$,
as in (\ref{eqn:schur}).
Then 
\begin{eqnarray}\label{eqn:blockdiag}
{\cal H} = P \begin{bmatrix}T_{11} & 0\\0 &T_{22} \end{bmatrix} P^{-1} , \qquad
\mbox{where}\quad
P = Q \begin{bmatrix}I&K\\0 &I\end{bmatrix},
\end{eqnarray}
and $K$ is the unique solution to the Sylvester equation $T_{11} K-K T_{22} = - T_{12}$ 
\cite[page 224]{Stewart.Sun.90}. By using the relation 
${\cal S}(\alpha_k) = I - 2 \aa_k ({\cal H}+\alpha_k I)^{-1}$, it can be readily seen that
the same matrix $P$, block diagonalizes ${\cal S}(\alpha_k)$ independently of $k$, that is
$$
\Sym_k =P\begin{bmatrix}T_{11(k)} & 0\\0 &T_{22(k)} \end{bmatrix}P^{-1}, 
{\rm\qquad for \quad every \quad} k.
$$
In particular, the second block column of $P$ determines a basis for the
left stable invariant subspace of ${\cal H}$ and ${\cal S}(\alpha_k)$. More precisely,
letting $Q=[Q_1, Q_2]$, then
$D_n(\Ham^*)=\Range(Q_1-Q_2 K^*)$ and
$D_n(\Sym^*)=\Range(Q_1-Q_2 K^*)$.


\begin{theorem}
With the notation above, let 
$[I ; X_0]=U_0R_0$ be the skinny QR decomposition of $[I;X_0]$, 
and assume that
$X_0$ is such that 
$$
d=\dist\left(D_n\left(\Ham^*\right),\Range\left(\begin{bmatrix}I\\X_0\end{bmatrix}\right)\right)<1 .
$$
If for any $k>0$, the matrix $M_k$ in the iteration (\ref{eqn:basic_rec}) is nonsingular,
then the associated iterate $X_k$ satisfies
%
%
\begin{equation}
\begin{aligned}
\dist\left(\Range\left(\begin{bmatrix}I\\X_+\end{bmatrix}\right),
\Range\left(\begin{bmatrix}I\\X_k\end{bmatrix}\right)\right)
\leq \gamma 
\left\|\prod_{i=k}^1 T_{22(i)}\right\|_2\left\|\prod_{i=1}^k T_{11(i)}^{-1}\right\|_2 ,
\end{aligned}
\end{equation}
where $\gamma = \frac{\|R_0^{-1}\|_2}{\sqrt{1-d^2}}\left( 1+\frac{\|T_{12}\|_F}{sep(T_{11},T_{22})}\right)$.
\end{theorem}

\begin{proof}
From (\ref{eqn:basic_rec}) and substituting $\Sym_k = Q T_{(k)} Q^*$, we obtain
\begin{eqnarray}
\begin{bmatrix}I\\X_{k}\end{bmatrix}M_k&=&\Sym_k \begin{bmatrix}I\\X_{k-1}\end{bmatrix}\\ 
Q^*\begin{bmatrix}I\\X_{k}\end{bmatrix}M_k&=&T_{(k)}Q^*\begin{bmatrix}I\\X_{k-1}\end{bmatrix} 
\end{eqnarray}
Recalling the blocking $Q=[Q_1, Q_2]$, let 
$\begin{bmatrix}V_k\\W_k\end{bmatrix}:=Q^*\begin{bmatrix}I\\X_{k}\end{bmatrix}$, so that
$$
\begin{bmatrix}V_k\\W_k\end{bmatrix}M_k=T_{(k)}\begin{bmatrix}V_{k-1}\\W_{k-1}\end{bmatrix}.
$$
Using the block diagonalization in  (\ref{eqn:blockdiag}) we obtain
$$
\begin{bmatrix}V_k-KW_k\\W_k\end{bmatrix}M_k=\begin{bmatrix}T_{11(k)} & 0\\0 &T_{22(k)} \end{bmatrix}
\begin{bmatrix}V_{k-1}-KW_{k-1}\\W_{k-1}\end{bmatrix}.
$$
Later in the proof we shall show that $V_{0}-KW_{0}$ is nonsingular. Under such assumption,
and since both $T_{11(k)}$ and $M_k$ are nonsingular as well, it follows from an
induction argument that  $V_{k}-KW_{k}$ is nonsingular. 
Therefore, recursively applying the same relation, we obtain
\begin{eqnarray}
W_k&=&T_{22(k)}W_{k-1}(V_{k-1}-KW_{k-1})^{-1}T_{11(k)}^{-1}(V_k-KW_k) \nonumber \\ 
& = & 
T_{22(k)}T_{22(k-1)}\cdots T_{22(1)} (V_0-KW_0)^{-1}T_{11(1)}^{-1}\cdots T_{11(k-1)}^{-1}T_{11(k)}^{-1}(V_k-KW_k)\nonumber\\
&=&\prod_{i=k}^1 T_{22(i)}\ (V_0-KW_0)^{-1}\prod_{i=1}^k T_{11(i)}^{-1}\begin{bmatrix}I&-K\end{bmatrix}
\begin{bmatrix}V_k\\W_k\end{bmatrix}.  \label{eqn:W}
\end{eqnarray}

The matrix $W_k$ is related to the distance of the two spaces of interest. Indeed,
let $[I;X_k]= U_k R_k$ be the skinny QR decomposition of $[I;X_k]$. Then using the
expression for the distance in \cite[section 2.6.3]{MR1417720}, we have
$$
\dist\left(\Range(Q_1),\Range\left(\begin{bmatrix}I\\X_k\end{bmatrix}\right)\right)=\|Q_2^*U_k\|_2
 =\left\|Q_{2}^*\begin{bmatrix}I\\X_k\end{bmatrix}R_k^{-1}\right\|_2=\|W_kR_k^{-1}\|_2 .
$$
Using  $\Range\left(\begin{bmatrix}I\\X_+\end{bmatrix}\right)=\Range(Q_1)$ and (\ref{eqn:W}) we obtain
\begin{eqnarray*}
\dist\left(\Range\left(\begin{bmatrix}I\\X_+\end{bmatrix}\right), 
\Range\left(\begin{bmatrix}I\\X_k\end{bmatrix}\right)\right)
&=& \|W_kR_k^{-1}\|_2\\
&\le & \gamma_0  \left\|\prod_{i=k}^1 T_{22(i)}\right\|_2\left\|\prod_{i=1}^k T_{11(i)}^{-1}\right\|_2,
\end{eqnarray*}
with $\gamma_0=\|(V_0-KW_0)^{-1}\|_2\|\begin{bmatrix}I&-K\end{bmatrix}\|_2$,
where in the last inequality we used the fact that the matrix
$\begin{bmatrix}V_k\\W_k\end{bmatrix}R_k^{-1}=Q^*\begin{bmatrix}I\\X_k\end{bmatrix}R_k^{-1}=Q^*U_k$ has
orthonormal columns.

We are left to estimate $\gamma_0$ and to ensure the nonsingularity of
$V_0-KW_0$.
Since $K$ is the solution to the Sylvester equation $T_{11} K-K T_{22} = - T_{12}$, it
follows
\begin{equation} \label{eqn:ImK}
\|\begin{bmatrix}I&-K\end{bmatrix}\|_2\leq 1+\|K\|_F\leq 1+\frac{\|T_{12}\|_F}{{\rm sep}(T_{11},T_{22})}.
\end{equation}
Let $Z=Q\begin{bmatrix}I\\-K^*\end{bmatrix}(I+KK^*)^{-\frac{1}{2}}$; clearly the columns of $Z$
are orthonormal and moreover, they span $D_n({\cal H}^*)$. Therefore, it holds that
(cf., e.g., \cite[Theorem 6.1]{MR1417720})
\begin{equation}
\begin{aligned}
d&=\dist\left(D_n(\Ham^*),\Range\left(\begin{bmatrix}I\\X_0\end{bmatrix}\right)\right)=\dist(D_n(\Ham^*),\Range(U_0))\\
&=\sqrt{1-\sigma_{\min}(Z^*U_0)}.
\end{aligned}
\end{equation}
Since $d<1$ by hypothesis, this relation shows that
the smallest singular value $\sigma_{\min}(Z^*U_0)$ is nonzero, and thus
$Z^*U_0$ is nonsingular.
From $U_0R_0=\begin{bmatrix}I\\X_0\end{bmatrix}$ and $\begin{bmatrix}V_0\\W_0\end{bmatrix}=Q^*\begin{bmatrix}I\\X_{0}\end{bmatrix}$, we obtain
\begin{eqnarray*}
V_0-KW_0&=& [I, -K] Q^* [I;X_0] = [I, -K] Q^* U_0 R_0 = (I+KK^*)^{\frac{1}{2}} (Z^* U_0) R_0.
\end{eqnarray*}
Since all three factors on the right are nonsingular, this shows that $V_0-KW_0$ is nonsingular;
moreover, using $\|(Z^* U_0)^{-1}\| = \sigma_{\min}(Z^* U_0)^{-1} = 1/\sqrt{1-d^2}$, we can write
\begin{eqnarray}\label{eqn:V0W0}
\|(V_0-KW_0)^{-1}\| \le \frac{\|R_0^{-1}\|}{\sqrt{1-d^2}} \|(I+KK^*)^{-\frac{1}{2}}\| ,
\end{eqnarray}
with $\|(I+KK^*)^{-\frac{1}{2}}\|\le 1$. Together with (\ref{eqn:ImK}), the estimate (\ref{eqn:V0W0})
bounds $\gamma_0$ from above, giving the final result.
\end{proof}


\bibliography{%
/home/valeria/Bibl/Biblioteca}


\end{document}